\documentclass[11pt]{article}
\usepackage{amsfonts}
\usepackage{amsmath,amsthm,mathtools,amscd,amssymb,mathrsfs,setspace}
\usepackage{latexsym, epsf, epsfig}
\usepackage{color}
\usepackage[hmargin=.95in,vmargin=.95in]{geometry}
\usepackage{graphicx}
\usepackage{hyperref}
\usepackage{cite}
\usepackage{multirow}
\usepackage{cancel}
\usepackage{float} 
\newcommand{\ds}{\displaystyle}

\newcommand{\nn}{\nonumber}

\newcommand{\cA}{{\mathcal{A}}}

\newcommand{\cE}{{\mathcal{E}}}

\newcommand{\bu}{\mathbf u}
\newcommand{\bv}{\mathbf v}
\newcommand{\bw}{\mathbf w}
\newcommand{\bF}{\mathbf F}

\newcommand{\bV}{\mathbf V}

\newcommand{\by}{\mathbf y}
\newcommand{\byt}{\widetilde{\mathbf y}}

\renewcommand{\vec}[1]{\mathbf{#1}}

\newcommand{\btau}{\boldsymbol{\tau}}
\newcommand{\bxi}{\boldsymbol{\xi}}
\newcommand{\bzeta}{\boldsymbol{\zeta}}
\newcommand\restri[1]{\left.{#1}\right|_{\Gamma_I}}
\newcommand\weakto\rightharpoonup

\theoremstyle{plain}
\newtheorem{theorem}{Theorem}[section]
\newtheorem{lemma}[theorem]{Lemma}
\newtheorem{proposition}[theorem]{Proposition}
\newtheorem{corollary}[theorem]{Corollary}

\newtheorem{definition}{Definition}
\theoremstyle{remark}
\newtheorem{remark}{Remark}[section]
\numberwithin{equation}{section} \numberwithin{theorem}{section}
\numberwithin{remark}{section} 

\begin{document}

\title{Uniqueness of Weak Solutions for Biot-Stokes Interactions}
 \author{{\small \begin{tabular}[t]{c@{\extracolsep{.8em}}c}
     George Avalos &     {  Justin T. Webster} \\
\it Univ. Nebraska-Lincoln  \hskip.2cm   & \hskip.4cm \it Univ. of Maryland, Baltimore County   \\
\it Lincoln, NE & \it Baltimore, MD\\
gavalos@math.unl.edu &    websterj@umbc.edu \\
\end{tabular}}}
\maketitle

\begin{abstract}
\noindent  We resolve the issue of uniqueness of weak solutions for linear, inertial fluid-poroelastic-structure coupled dynamics. The model comprises a 3D Biot poroelastic system coupled to a 3D incompressible Stokes flow via a 2D interface, where kinematic, stress-matching, and tangential-slip conditions are prescribed. Our previous work provided a construction of weak solutions, these satisfying an associated finite energy inequality. However, several well-established issues related to the dynamic coupling, hinder a direct approach to obtaining uniqueness and continuous dependence. In particular, low regularity of the hyperbolic (Lam\'e) component of the model precludes the use of the solution as a test function, which would yield the necessary a priori estimate. In considering degenerate and non-degenerate cases separately, we utilize two different approaches. In the former, energy estimates are obtained for arbitrary weak solutions through a systematic decoupling of the constituent dynamics, and well-posedness of weak solutions is inferred. In the latter case, an abstract semigroup approach is utilized to obtain uniqueness via a precise characterization of the adjoint of the dynamics operator. The results here can be adapted to other systems of poroelasticity, as well as to the general theory of weak solutions for hyperbolic-parabolic coupled systems. 
\vskip.25cm
\noindent Keywords: {fluid-poroelastic-structure interaction, poroelasticity, Biot model, filtration problem, Beavers-Joseph-Saffman, semigroup methods}
\vskip.25cm
\noindent
{\em 2020 AMS MSC}: 74F10, 76S05, 35M13, 76M30, 35D30 
\vskip.4cm
\noindent Acknowledgments: The first author was partially supported by NSF-DMS 1907823; the second author was partially supported by NSF DMS-2307538.
\end{abstract}

\section{Introduction}

In this work we provide a complementary uniqueness result to the results in \cite{oldpaper}, where finite-energy weak solutions were constructed for an inertial 3D Biot-Stokes interaction. In that previous work, a dynamics operator is proposed which generates a strongly continuous semigroup of contractions. The model of interest is the fully dynamic poroelastic Biot dynamics \cite{Sanchez-Palencia,show2000,indiana}, coupled across a  2D boundary interface with the incompressible Stokes equations, via the so called Beavers-Joseph-Saffman (slip-type) conditions \cite{showfiltration}. The method in \cite{oldpaper} uses a Green's mapping approach to eliminate the pressure, as in \cite{avalos,avalos*}. With a $C_0$-semigroup in hand, strong and mild solutions are obtained, including corresponding uniqueness and continuous dependence results for those solution notions. The existence of weak solutions was then achieved through a density argument for solutions emanating from the domain of the semigroup generator. This approach also permitted the construction of weak solutions in the degenerate case of incompressible Biot constituents through a secondary limiting argument. 

However, owing to the low regularity of the hyperbolic-like dynamics at the interface in the Biot-Stokes model, weak solutions are not themselves viable test functions. This is a classical issue, noted for the wave equation \cite{evans}. Moreover, the standard regularization procedures (e.g., \cite{temam}) are complicated by the nature of the coupling on a portion of the interface. This is a central issue  \cite{pata} for hyperbolic-parabolic coupled systems, of course including fluid-structure interactions. This issue for uniqueness of weak solutions for coupled fluid-poroelastic dynamics has been alluded to in several previous works \cite{multilayered,oldpaper,rectplate,sunny3}, and has remained heretofore an open question. For linear dynamics, uniqueness of weak solutions is essential for later tackling  physically-motivated models  that incorporate elastic and geometric nonlinearities; even in the case of uncoupled Biot's dynamics, uniqueness is a complex issue of recent interest \cite{bw,bmw,bgsw,multilayered,cao}.

In the work at hand, we fully resolve the uniqueness of weak solutions for this filtration problem. This requires two separate proof methodologies: one for the non-degenerate case of compressible Biot constituents (with an underlying semigroup), and another when the Biot storage coefficient vanishes (degenerate parabolic component). In particular, the time-regularity of the  fluid-content $\zeta$, given in terms of the Biot pressure $p$ and  elastic displacement $\bu$ as the sum $\zeta=[c_0p+\alpha \nabla \cdot \bu]$, plays a central role. When $c_0=0$, the dynamics no-longer admit a proper semigroup, as the evolution variable $p$ becomes degenerate (lacking a time derivative). Yet when $c_0=0$, we directly obtain information on the regularity of $\alpha\nabla \cdot \bu_t$ from the pressure equation, which helps in properly characterizing the hyperbolic regularity of the  displacement $\bu$ and velocity $\bu_t$ variables. When $c_0>0$, the dynamics are non-degenerate and semigroup methods are applicable, but the information on the fluid content---gleaned from conservation of mass---comes only through the aforementioned sum $\zeta$, which hinders the previous argument. 

\subsection{Approaches to Uniqueness}
In the non-degenerate case $c_0>0$, we will avoid issues of low regularity of weak solutions by adopting an approach inspired by J. Ball's characterization of semigroup weak solutions in \cite{ball}, valid for dynamics which admit a semigroup generator. This is done in part by finding the explicit representation of the adjoint of the generator. Subsequently, by showing that weak solutions for the Biot-Stokes dynamics in the physical, finite energy, sense also constitute weak solutions in the adjoint sense of Ball, we can appeal to the uniqueness provided \cite{ball} by the semigroup established in \cite{oldpaper}. This task of finding the adjoint of the Biot-Stokes semigroup generator, however, is  nontrivial; the dynamics generator in \cite{oldpaper} is quite complex, involving several constituent operators and Green's mappings used to eliminate the Stokes pressure. Characterizing the domain of this matrix of differential and nonlocal operators is a challenging endeavor, one which we relegate to the Appendix. Once the adjoint characterization is obtained,  we then work to  {\em verify that arbitrary weak solutions}, as defined in \cite{oldpaper} and adapted from \cite{multilayered}, are indeed {\em semigroup solutions}. Thus, uniqueness follows from the established Biot-Stokes semigroup via \cite{ball}. 

In the degenerate case $c_0=0$, we are motivated by the approach toward uniqueness in \cite{bmw}. In particular, we focus on the null-data equations (as the problem as linear) and look toward component-wise regularity, using a systematic decoupling approach. Using zero data, we utilize various selections of test functions to extract additional regularity, and then apply various interpolation-type results. In this way, we justify each component of the energy identity, piece-by-piece, and justify cancellation of low regularity terms appearing on the interface (as well as in the interior). One of the central challenges in this analysis is the inability to separate terms which formally ``cancel out" in the energy relation, owing to the nature of the coupling. In this argument, we critically use the result of Temam \cite{temam1} for hyperbolic-like systems, which boosts weak solutions having $L^{2}(0,T;V)\cap H^{1}(0,T;H)\cap H^2(0,T;V')$ regularity to the regularity class of semigroups, namely $C([0,T],V)\cap C^1(0,T;H)$. 

In both cases, once uniqueness is established, continuous dependence follows immediately, since the previously constructed weak solutions in \cite{oldpaper} obey an energy estimate. 

\subsection{Notation}

We will work in the $L^2(U)$ framework, where $U \subseteq \mathbb R^n$, for $n=2,3$ a given spatial domain. We denote standard $L^2(U)$ inner products by $(\cdot,\cdot)_U$.  Classical Sobolev spaces of the form $H^s(U)$ and $H^s_0(U)$ (along with their duals) will be defined in the standard way \cite{kesavan}, with the $H^s(U)$ norm denoted by $||\cdot||_{s,U}$, or just $||\cdot||_{s}$ when the context is clear. For a Banach space $Y$ we denote its (Banach) dual as $Y'$, and  the associated duality pairing as $\langle \cdot, \cdot\rangle_{Y'\times Y}$. The notation $\mathscr{L}(X,Y)$ denotes the (Banach) space of bounded linear operators from Banach spaces $X$ to $Y$, and we write $\mathscr{L}(Y)$ when $X=Y$. We denote $\vec{x} = (x_1,x_2,x_3) \in \mathbb{R}^3$, with associated spatial differentiation by $\partial_i$. For estimates, we will use the notation $A\lesssim B$ to mean there exists a constant $c$ for which $A \le c B$.

\subsection{Biot-Stokes Model and Spatial Domain} 
We consider the fluid-poroelastic-structure interaction (FPSI) model as presented in \cite{oldpaper}. The Biot displacement variable, $\bu$, is associated with a Biot pressure, $p$; both are defined on a 3D domain $\Omega_b$. We take the Biot domain to be fully-saturated, in the sense of \cite{coussy,show2000}, and we assume the porous matrix is isotropic and homogeneous. The function $\bF$ represents a volumetric force, and $S$ a diffusive source. Therewith, a poroelastic system on $\Omega_b$ is modeled by the following equations:
\begin{equation}\label{biot1}
\begin{cases}
    \rho_b \vec{u}_{tt} - \mu\Delta\vec{u} - (\lambda + \mu)\nabla(\nabla\cdot\vec{u}) + \alpha \nabla p = \vec{F}_b, &\text{ in } \Omega_b \times (0,T),\\
    [c_0p + \alpha \nabla \cdot \vec{u}]_t - \nabla \cdot [k\nabla p] = S, &\text{ in } \Omega_b \times (0,T).
\end{cases}
\end{equation}
 Both $\bu$ and $p$ are homogenized quantities in the poroelastic region \cite{coussy,Sanchez-Palencia}. We take  $\rho_b\ge 0$ to be the density corresponding to the poroelastic region; we consider here the {\em inertial case} when $\rho_b>0$. The parameters $\lambda,\mu >0$ are the standard Lam\'e coefficients of elasticity \cite{ciarletbook}. The coupling constant $\alpha>0$ is known as the {\em Biot-Willis} coupling constant \cite{biot2,coussy}. The storage coefficient $c_0\ge 0$ represents the net  compressibility of the fluid-solid matrix; when $c_0=0$ we say that {\em the system has incompressible constituents}. The quantity  \begin{equation}\label{biotfluid}
    \zeta = c_0p + \alpha \nabla \cdot\vec{u}
\end{equation}
is known as the {\em fluid content} of the Biot system. Finally, $k \in \mathbb R_+$ represents the scalar permeability of the porous matrix. 
  
We consider a free fluid velocity, $\bv$ with pressure, $p_f$, defined on a 3D domain $\Omega_f$:
\begin{equation}\label{stokes1}
    \rho_f \vec{v}_t - 2\nu \nabla \cdot \vec{D}(\vec{v}) + \nabla p_f = \vec{F}_f,\quad \quad \nabla\cdot\vec{v} = 0, \quad  \text{ in } \Omega_f \times (0,T).
\end{equation}
Here, $\rho_f$ denotes the density of the fluid at a reference pressure and $\nu$ denotes the shear viscosity of the fluid. 
We assume that the elastic  stress  $\sigma^E(\vec{u})$ obeys the linear strain-displacement law \cite{kesavan,show2000} given by
\begin{equation}\label{elasticstress}
    \sigma^E(\vec{u}) = 2\mu\vec{D}(\vec{u}) + \lambda(\nabla\cdot\vec{u})\vec{I},
\end{equation}
where $\vec{D}(\vec{u}) = \frac{1}{2}(\nabla\vec{u} + (\nabla\vec{u})^T)$ is the symmetrized gradient \cite{kesavan,temam}. 
We then denote the total poroelastic and fluid stress tensors as
\begin{align*}
    \sigma_b = \sigma_b(\vec{u},p) &= \sigma^E(\vec{u}) - \alpha p\vec{I}, \quad \quad \sigma_f = \sigma_f(\vec{v},p_f) = 2\nu\vec{D}(\vec{v}) - p_f\vec{I},
\end{align*}

{In this note, we  take $\rho_b,\rho_f,\nu,\alpha, \mu>0$, and $c_0, \lambda \ge 0$.}
In line with \cite{oldpaper,multilayered}, we focus on the spatial configuration of two stacked  boxes in 3D.  We will effectively identify the lateral sides (i.e., in the $x_1$ and $x_2$ directions), producing a flat, laterally toroidal domain. By considering lateral periodicity we may focus on the specific challenges associated to the interface dynamics; future work will address full, physical boundary conditions, including mixed boundary conditions with corners (as in \cite{avalos,avalos*}), without invoking lateral periodicity. 

\begin{center}
    \includegraphics[width=.7\textwidth]{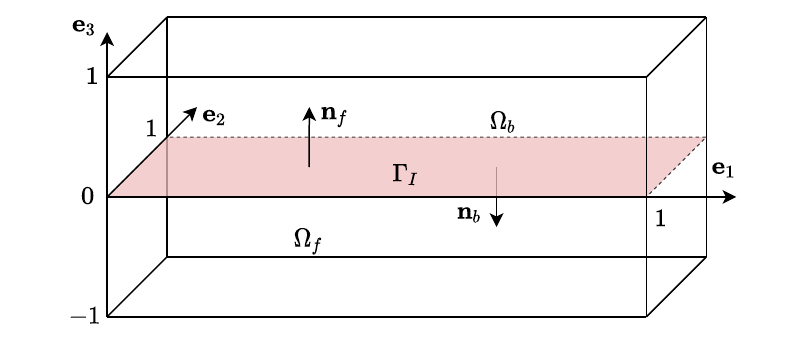}
\end{center}

\noindent Now---and for the remainder of the paper---we take $\Omega_b \equiv (0,1)^3$ to be a fully-saturated poroelastic structure. The dynamics on $\Omega_b$ are modeled by the Biot equations, described above. 
Let $\Omega_f \equiv  (0,1)\times (0,1) \times (0,-1)$ be a region adjacent to $\Omega_b$.
The two regions share an interface $\Gamma_I = \partial\Omega_b \cap \partial \Omega_f = (0,1)^2$. The normal vectors going out of the Biot and Stokes regions are denoted by $\vec{n}_b$ and $\vec{n}_f$, respectively. Then we orient the interface, so that on $\Gamma_I$, $\vec{n}_f = -\vec{n}_b = \vec{e}_3$. Denote the Biot and Stokes boundaries by
\begin{equation} \Gamma_b = \{\phi \in \partial \Omega_b~:~x_3=1\}~\text{ and }~ \Gamma_f = \{\phi \in \partial \Omega_f~:~x_3=-1\}.\end{equation} 
We consider, for these upper and lower boundary components, the homogeneous conditions
\begin{equation}\label{BC1}
    \vec{u} = \vec{0} ~~\text{ and }  ~~p = 0 \text{ on }\Gamma_b, \quad \text{ and } ~~\vec{v} = \vec{0} \text{ on } \Gamma_f, \quad \forall\, t \in (0,T).
\end{equation}
As previously mentioned, we consider periodic conditions on the lateral boundaries, collectively defined as $\Gamma_{lat}$, i.e., as the union of the four lateral faces in the $x_1$ and $x_2$ directions. In other words, we identify face quantities that oppose one another in $x_1$ or $x_2$ {\em whenever Sobolev traces are appropriately defined}.
At the interface $\Gamma$, we enforce stress-matching and velocity-matching type conditions, as well as a tangential slip condition. These are written in strong form as:
\begin{align}
    -k\nabla p\cdot\vec{e}_3 &= \vec{v} \cdot \vec{e}_3 - \vec{u}_t \cdot \vec{e}_3, \label{IC1}\\
    \beta(\vec{v} - \vec{u}_t)\cdot \btau &= -\btau \cdot \sigma_f \vec{e}_3,\label{IC2}\\
    \sigma_f\vec{e}_3 &= \sigma_b\vec{e}_3,\label{IC3}\\
    p &= -\vec{e}_3 \cdot \sigma_f\vec{e}_3 
    .\label{IC4}
\end{align}
We use the notation of  $\btau$ for the tangent vectors on $\Gamma_I$, i.e., as $\btau = \vec{e}_i$ for $i = 1,2$. The tangential slip condition is known as the {\it Beavers-Joseph-Saffman} condition \cite{mikelicBJS}, and the parameter $\beta>0$ is called the {\it slip length} (see also \cite{multilayered, yotovBJS}). 

In line with \cite{oldpaper}, we utilize natural energies which arise by: formally testing the above Biot-Stokes system \eqref{biot1} with $(\vec{u}_t, p)$ and \eqref{stokes1} with $\vec{v}$, assuming sufficient regularity of the solution to justify integration by parts, and invoking boundary conditions in \eqref{IC1}--\eqref{IC3} at the interface $x_3=0$. Defining the total trajectory energy $e(t)$ as \begin{equation} \label{energy} e(t) \equiv  \frac{1}{2}\Big[\rho_b\|\vec{u}_t\|_{0,b}^2 + \|\vec{u}\|_E^2 + c_0\|p\|_{0,b}^2 + \rho_f\|\vec{v}\|_{0,f}^2\Big],\end{equation} and a dissipator $d_0^t$ as \begin{equation}\label{dissipator} d_0^t \equiv \int_0^t\big[k\|\nabla p\|_{0,b}^2 
+ 2\nu \|\vec{D}(\vec{v})\|_{0,f}^2 
+ \beta \|(\vec{v}-\vec{u}_t)\cdot\boldsymbol{\tau} \|_{\Gamma_I}^2 \big] d\tau,\end{equation}
we obtain the formal energy balance (taking all sources to be zero):
\begin{equation} \label{eident} e(t) +d_0^t = e(0).\end{equation}
Above, we have denoted the norm \begin{equation}\|\vec{u}\|^2_E \equiv  (\sigma^E(\vec{u}),\vec{D}(\vec{u}))_{\mathbf L^2(\Omega)},\end{equation} which is equivalent to the standard $\mathbf H^1(\Omega_b)$ norm by Korn's inequality and Poincar\'e's inequality \cite{kesavan,temam}.
This identity holds for smooth solutions \cite{oldpaper} (e.g., in the domain of the semigroup generator) and is extended to semigroup solutions, for finite energy data.  
With reference to (2.1)--(2.9) below, we also refer to the {\em energy inequality} as
\begin{align} \label{eidentt}
e(t) + d_0^t \le & ~ e(0) \\ \label{eidenttt}
e(t) + d_0^t \lesssim & ~ e(0) + \int_0^t\big[||\mathbf F_f(\tau)||^2_{\vec{L}^2(\Omega_b)}+||\mathbf F_b(\tau)||^2_{[\mathbf H^1_{\#,*}(\Omega_f)\cap \mathbf V]'}+||S(\tau)||_{[H^1_{\#,*}(\Omega_b)]'}^2 \big]d\tau
\end{align}
in the first case when $\mathbf F_f = \mathbf F_b = S \equiv 0$, and in the latter case when the sources are present---the function spaces being defined in detail in Section \ref{opssp}. 
The energy inequality holds for weak solutions through an approximation argument in \cite{oldpaper}. In this paper, we will show that such weak solutions are in fact unique, and hence the energy estimate  in \eqref{eidenttt} holds for the unique weak solution. Then, by linearity, superposition gives that the weak problem is well-posed, of course making the appropriate assumptions on the data.

\section{Relevant Definitions and Past Result}\label{mainresults}
The boundary value problem of interest, written in divergence form here, is:
\begin{align}\label{systemfull}
    \rho_b \vec{u}_{tt} - \nabla\cdot\sigma^E(\vec{u}) + \alpha \nabla p =&~ \vec{F}_b, &\text{ in } \Omega_b \times (0,T),\\ \label{systemfull0}
    [c_0p + \alpha \nabla \cdot \vec{u}]_t - \nabla \cdot [k\nabla p] =&~ S, &\text{ in } \Omega_b \times (0,T),\\
    \rho_f \vec{v}_t - 2\nu \nabla \cdot \vec{D}(\vec{u})  + \nabla p_f =&~ \vec{F}_f,\quad \quad \nabla\cdot\vec{v} = 0, &  \text{ in } \Omega_f \times (0,T).\label{systemfull1}
\end{align}
Boundary conditions at $x_3=\pm 1$ are given by:
\begin{equation}\label{uplow}
    \vec{u} = \vec{0} ~~\text{ and } ~~ p = 0~ \text{ on }~\Gamma_b, \quad ~\text{ and } ~~\vec{v} = \vec{0} ~\text{ on }~ \Gamma_f, ~~~t \in (0,T).\end{equation}
Interface conditions on $\Gamma_I\times (0,T)$ are:
\begin{align}
    -k\nabla p\cdot\vec{e}_3 &= \vec{v} \cdot \vec{e}_3 - \vec{u}_t \cdot \vec{e}_3, \label{IC1*}\\
    \beta(\vec{v} - \vec{u}_t)\cdot \btau &= -\btau \cdot \sigma_f \vec{e}_3,\label{IC2*}\\
    \sigma_f\vec{e}_3 &= \sigma_b\vec{e}_3,\label{IC3*}\\
    p &= -\vec{e}_3 \cdot \sigma_f\vec{e}_3. 
    \label{IC4*}
\end{align}
Periodic conditions are taken on the lateral faces $\Gamma_{lat}\times (0,T)$. Initial conditions are prescribed for the states
\begin{equation}\label{endfull}\vec{y}(0) = [\vec{u}(0),\vec{u}_t(0),p(0),\vec{v}(0)]^T,\end{equation}
in line with the energy $e(t)$ as defined in \eqref{energy}.

\subsection{Operators and Spaces}\label{opssp}

For either domain $\Omega_i$ ($i=b,f$), let us introduce the notation $H^s_\#(\Omega_i)$ for the space of all functions from $H^s(\Omega_i)$ that are 1-periodic in directions $x_1$ and $x_2$:
$$    {H}^s_\#(\Omega_i) := \{f \in {H}^s(\Omega_i)~:~ \gamma_j[f]|_{x_1 = 0} = \gamma_j[f]|_{x_1 = 1},~\gamma_j[f]|_{x_2 = 0} = \gamma_j[f]|_{x_2 = 1} \text{ for } j=0,1,...,s-1\},$$
where we have denoted by $\gamma_j$ the $j$th standard trace map \cite{kesavan}. We take the analogous definition for the vector-valued functions in $\mathbf H^s_\#(\Omega_i)$. 
Then let
\begin{align*}
    \vec{U} &\equiv \{\vec{u} \in \vec{H}_\#^1(\Omega_b)~:~~ \vec{u}\big|_{\Gamma_b} = \vec{0}\},\\
    \vec{V} &\equiv \{\vec{v} \in \vec{L}^2(\Omega_f) ~:~~ \text{div}~ \vec{v} \equiv 0 \text{ in }\Omega_f; \, ~~ \vec{v}\cdot\vec{n} \equiv 0\text{ on } \Gamma_f\},\\
    X &\equiv  \vec{U} \times \vec{L}^2(\Omega_b) \times L^2(\Omega_b) \times \vec{V},
\end{align*}
the latter $X$ being the finite energy space for the dynamics. 

For ease, we will also use the notation below $\vec{H}^1_{\#,\ast}(\Omega_f)$ and ${H}^1_{\#,\ast}(\Omega_b)$ where the subscript $*$ on denotes a zero Dirichlet trace on $\Gamma_f$ and $\Gamma_b$, resp.; that is, on the non-interface face. In this way we can identify $\mathbf U \equiv \mathbf{H}^1_{\#,\ast}(\Omega_b)$.
The space $H^1_{\#,*}(\Omega_b)$ is topologized via the Poincar\'e inequality \cite{kesavan,bgsw} with the gradient norm; thus $$||\cdot||_{H^1_{\#,*}(\Omega_b)} \equiv ||\nabla \cdot||_{L^2(\Omega_b)}.$$
The topology of $\mathbf U$ is induced by the norm $||\cdot||_E$ introduced before through  \begin{equation}\label{bilform} a_E(\cdot,\cdot)=(\sigma^E(\cdot),\mathbf D(\cdot))_{\mathbf L^2(\Omega_b)}.\end{equation} This norm is equivalent to the full $\mathbf H^1(\Omega_b)$ norm in $\mathbf U$  \cite{kesavan}. 
Moreover, for the following differential actions, we set: \begin{equation}\label{ops}\cE_0 = -\rho_b^{-1}\nabla\cdot\sigma^E~\text{ and }~ A_0 = -c_0^{-1}k\Delta.\end{equation}
When $c_0,\rho_b,\rho_f>0$, we introduce some equivalent topologies for the inner-product on the finite-energy space $X$:
\begin{align}\label{innerX}
(\mathbf y_1, \mathbf y_2)_X = a_E(\bu_1,\bu_2)+\rho_b(\bw_1,\bw_2)_{\mathbf L^2(\Omega_b)}+c_0(p_1,p_2)_{L^2(\Omega_b)}+\rho_f(\mathbf v_1,\mathbf v_2)_{\mathbf L^2(\Omega_f)}.
\end{align}
 We take $\vec{w} = \vec{u}_t$  to be the elastic state variable for velocity. 
We then consider a Cauchy problem which captures the dynamics of the Biot-Stokes system. Namely: find $\vec{y} = [\vec{u},\vec{w},p,\vec{v}]^T \in \mathcal{D}(\cA)$ (resp. $X$) such that
\begin{equation}\label{cauchy}
   \dot{ \vec{y}}(t) = \cA \vec{y}(t)+\mathcal F(t); \quad \vec{y}(0) = [\mathbf u_0,\mathbf u_1,p_0,\mathbf v_0]^T \in \mathcal{D}(\cA) \text{ (resp. } X),
\end{equation}
where $\mathcal F(t) = [\mathbf 0, \mathbf F_b(t), S(t), \mathbf F_f(t)]^T$ (of appropriate regularity).
The action of the  operator $\cA:\mathcal{D}(\cA)\subset X \to X$ is provided by
\begin{equation}\label{A*}
    \cA \equiv \begin{pmatrix}
        0 & \vec{I} & 0 & 0\\
        -\cE_0 & 0 & -\alpha\rho_b^{-1} \nabla & 0\\
        0 & -\alpha c_0^{-1} \nabla\cdot & -A_0 & 0\\
        0 & 0 & \rho_f^{-1}G_1 & \rho_f^{-1}[\nu\Delta + G_2 + G_3]
    \end{pmatrix}.
\end{equation}
The operators $G_1, G_2,$ and $G_3$ above are certain Green's maps which provide the elimination of the pressure (and carry boundary information) and will be defined properly in Section \ref{operator}.

We will need some dual trace spaces, associated to the spaces $H^s_{\#}(\Omega_i)$ for $i=b,f$. We first note that the image of the Dirichlet trace map $\gamma_0: H^1(\Omega_i) \to H^{1/2}(\Omega_i)$ is well-defined by the Sobloev Trace theorem, as are subsequent restrictions to any faces of the cubic geometry of $\Omega_b$ or $\Omega_f$ \cite{Grisvard}. To preserve and utilize the associated duality pairing with trace spaces such as $\gamma_0[H^{1}_{\#}(\Omega_i)]$---for instance to characterize higher order traces---we now explicitly define $H_{\#}^{-1/2}(\partial \Omega_f)$. The definition will then permit  such negative traces on any face via restriction (in the sense of distributions); an analogous definition will also hold for spaces defined on $\Omega_b$. For example, let $\Gamma_a \subset \Gamma_{\text{lat}}$ and $\Gamma_b \subset \Gamma_{\text{lat}}$ generically represent diametrically opposed faces of $\Omega_b$ or $\Omega_f$---namely $(a,b)=(\{x_1=0\}, \{x_1=1\})$ or $(\{x_2=0\}, \{x_2=1\})$. Then
\begin{align}\label{tracespaces}
H^{-1/2}_{\#}(\partial \Omega_f)=  \Big\{g \in H^{-1/2}(\partial \Omega_f)~:~g \big|_{\Gamma_a} = g\big|_{\Gamma_b} \in H^{-1/2}\big((0,1)\times (-1,0)\big)~\text{ for each pair}~\Gamma_a, \Gamma_b\Big\}.
\end{align} Likewise, we obtain the dual space $H^{-1/2}_{\#}(\partial \Omega_b)$. Mutatis mutandis, we obtain the spaces $H^{-r}_{\#}(\partial \Omega_i)$, for $r=1$ and $r=3/2$. 

\subsection{Definition of Weak Solutions and Discussion}
We will now define weak solutions for all $c_0\ge 0$ and we begin with some spaces:
\begin{align*}
    \mathcal{V}_b &= \{\vec{u} \in L^\infty(0,T;\vec{U}) ~:~ \vec{u} \in W^{1,\infty}(0,T;\vec{L}^2(\Omega_b))\}, \\
    \mathcal{Q}_b &= \left\{p \in L^2(0,T; H^1_{\#,*}(\Omega_b))~:~c_0^{1/2}p \in L^\infty(0,T;L^2(\Omega_b))\right\}, \\
    \mathcal{V}_f &= L^\infty(0,T;\vec{V}) \cap L^2(0,T; \vec{H}^1_{\#,\ast}(\Omega_f) \cap \vec{V}).
\end{align*}
The {weak solution space} is then
\newcommand\Vsol{\mathcal{V}_{\text{sol}}}
\newcommand\Vtest{\mathcal{V}_{\text{test}}}    
\[
    \mathcal{V}_\text{sol} = \mathcal{V}_b\times \mathcal{Q}_b\times\mathcal{V}_f,~\text{ for } [\bu,p,\bv]^T
\]
and the {test space} is
\[
    \mathcal{V}_\text{test} = C^1_0([0,T); \vec{U} \times  H^1_{\#,*}(\Omega_b) \times (\vec{H}^1_{\#,\ast}(\Omega_f)\cap \vec{V})),
\]
where we use the  notation $C^1_0$ for continuously differentiable functions with compact support on $(-\epsilon,T)$ for all $\epsilon > 0$. 
We introduce some further  notation. Let
 $((\cdot, \cdot))_{\mathscr O}$ denote an inner-product on $L^2(0,T;L^2(\mathscr O))$, for a spatial domain $\mathscr O$. Similarly we will denote $(\langle \cdot,\cdot \rangle)_{\mathscr O}$ below as the $L^2(0,T;[W(\mathscr O)]') \times L^2(0,T;W(\mathscr O))$ pairing, for the appropriate  pairing between the Banach space $W(\mathscr O)$ and its dual. 
We also suppress the explicit expression of trace operators below, noting instead boundary portion restrictions from subscripts. We will consider initial conditions of the form:
     \begin{equation}\label{weakICs}
        \vec{u}(0) = \vec{u}_0, ~~ \vec{u}_t(0) = \vec{u}_1, ~~ [c_0p + \alpha\nabla\cdot\vec{u}](0) = d_0, ~~ \vec{v}(0) = \vec{v}_0.
    \end{equation}
   Finally, we consider source data of regularity (at least): 
    $$\mathbf F_f \in L^2\big(0,T;(\vec{H}^1_{\#,\ast}(\Omega_f)\cap \vec{V})'\big),~~\mathbf F_b \in L^2(0,T;\vec{L}^2(\Omega_b)),~~\text{and}~~S \in L^2(0,T; (H^{1}_{\#,*}(\Omega_b))').$$
 
Then we may define weak solutions.
\begin{definition}\label{weaksols}
    We say that $[\vec{u},p,\vec{v}]^T \in \mathcal{V}_\text{sol}$ is a weak solution to \eqref{systemfull}--\eqref{endfull} if: \\[.2cm] (1) 
    For every test function $[\bxi,q_b,\bzeta]^T \in \mathcal{V}_\text{test}$ the following identity holds:
    \begin{align}\label{weakform}
        -&~ \rho_b((\vec{u}_t,\bxi_t))_{\Omega_b} + ((\sigma_b(\vec{u},p), \nabla \bxi))_{\Omega_b} - ((c_0 p + \alpha\nabla\cdot\vec{u}, \partial_t q_b))_{\Omega_b} + ((k\nabla p,\nabla q_b))_{\Omega_b} \nn\\
        &- \rho_f((\vec{v},\bzeta_t))_{\Omega_f} + 2\nu((\vec{D}(\vec{v}),\vec{D}(\bzeta)))_{\Omega_f} + (({p,(\bzeta-\bxi)\cdot\vec{e}_3}))_{\Gamma_I} - (({\vec{v}\cdot\vec{e}_3,q_b}))_{\Gamma_I} \nn\\
        &- (({\vec{u}\cdot\vec{e}_3,\partial_tq_b}))_{\Gamma_I} + \beta (({\vec{v}\cdot\btau,(\bzeta - \bxi)\cdot\btau}))_{\Gamma_I} + \beta(({\vec{u}\cdot\btau, (\bzeta_t - \bxi_t)\cdot\btau}))_{\Gamma_I} \nn\\
        =&~ \rho_b(\mathbf u_1,\bxi)_{\Omega_b}\big|_{t=0} + (d_0,q_b)_{\Omega_b}\big|_{t=0} + \rho_f(\vec{v}_0,\bzeta)_{\Omega_f}\big|_{t=0} - ({\vec{u}_0\cdot\vec{e}_3,q_b})_{\Gamma_I}\big|_{t=0} \nn\\
        &~+ \beta({\vec{u}_0\cdot\btau, (\bzeta - \bxi)\cdot\btau})_{\Gamma_I}\big|_{t=0} + ((\vec{F}_b,\bxi))_{\Omega_b} + (\langle S,q_b\rangle)_{\Omega_b} + (\langle \vec{F}_f,\bzeta\rangle)_{\Omega_f},
    \end{align}
 for $\btau=\mathbf e_i,~i=1,2$.
    \\[.2cm] (2) The function $\mathbf u \in \mathcal V_b$ has the additional property that $\gamma_0[\mathbf u]_t \cdot \btau \in L^2(0,T;L^2(\Gamma_I))$ for $\btau=\mathbf e_i,~i=1,2$.
\end{definition}
\begin{remark} In the motivating works \cite{oldpaper,multilayered} the analogous notions of weak solution did not include point (2) above. In those works, the constructed weak solution indeed possessed  property (2), but it was not included as part of the definition. As we shall see below, it will be necessary to include this in the definition for the uniqueness argument in the case $c_0=0$.
\end{remark}

In the above formulation, certain time differentiations are pushed to the time-dependent test function, as is a common treatment of hyperbolic problems. It is worth noting that the weak formulation circumvents a priori trace regularity issues at the interface associated to $\gamma_0[\bu_t]$ through this temporal integration by parts  (in the normal and tangential components of $\bu_t$ on $\Gamma_I$). We note that weak solutions have only that $\bu_t \in L^{\infty}(0,T;\mathbf L^2(\Omega_b))$ for the given interior regularity.  {\em Natural} Neumann-type conditions in \eqref{IC1*}--\eqref{IC4*}  are unproblematic, as such traces do not appear in the weak formulation and will be a posteriori justified as elements of $H^{-1/2}(\Gamma_I)$-type spaces. 
Yet, {\em for a given weak solution}, we can infer some regularity of temporal derivatives in the statement of \eqref{weakform}, {a posteriori}. For instance, the terms  $\bu_{tt}$ and $[c_0p+\nabla \cdot \bu]_t$ in the interior, as well as $\gamma_0[\bu]_t \cdot \mathbf e_3$, will exist as elements of certain dual spaces. However, characterizing a space of distributions for these time derivatives is not straightforward, owing to the coupled nature of the problem and its weak formulation in \eqref{weakform} above. The term $\gamma_0[(\bv-\bu_t)\cdot \btau]$ will be a properly defined element of $L^2(0,T;L^2(\Gamma_I))$ for any weak solution.

As a final comment, we mention that the classical reference \cite{ball} includes an abstract construction of weak solutions from any given semigroup, as well as uniqueness of said weak solutions; however, these weak solutions do not correspond to any energy estimates or a weak form predicated on symmetric bilinear forms. In the semigroup case here $c_0>0$, we will utilize the insights from \cite{ball} to obtain uniqueness; indeed, we show that weak solutions as defined here are in fact weak solutions in the sense of the semigroup, and are thus unique. When $c_0=0$, we do not have an underlying semigroup, and thus must manifest a different methodology. 

\subsection{Definition of the Semigroup Generator $\mathcal A$:~ $c_0>0$}\label{operator}
In \cite{oldpaper} it is shown that the operator $\mathcal A$ described above is a generator of a $C_0$-semigroup on $X$. Here we recall this operator precisely. This is necessary for Section \ref{ballproof}  below, as the aforesaid generator (and its adjoint) play the central role in the uniqueness argument for $c_0>0$. 

To define the dynamics operator we eliminate the Stokes pressure $p_f$ via elliptic theory, namely through appropriately defined Green's mappings \cite{redbook}. 
We consider a fluid-pressure sub-problem as in \cite{oldpaper}. 
We have the Neumann and Dirichlet Green's maps $N_f$ and $D_I$ by
\[
    \phi \equiv N_fg \iff 
    \begin{cases}
        \Delta \phi = 0 & \text{in} ~\Omega_f;\\
        \partial_{\vec{e}_3} \phi = g &  \text{on} ~ \Gamma_f;\\
        \phi = 0 &  \text{on} ~ \Gamma_I;
    \end{cases}
    \quad \quad
    \psi \equiv D_I h \iff
    \begin{cases}
        \Delta \psi = 0 &  \text{in} ~\Omega_f;\\
        \partial_{\vec{e}_3} \psi = 0 &  \text{on} ~ \Gamma_f;\\
        \psi = h &  \text{on} ~\Gamma_I;
    \end{cases}
\]
where we have suppressed the lateral periodicity implicit in the formulations on $\Gamma_{lat}$.
The maps $N$ and $D$ are well-defined via classical elliptic theory for all data classes \cite{redbook,lionsmag}:
\begin{align*}
    N_f &\in \mathcal{L}(H^s(\Gamma_f),  H^{s + \frac{3}{2}}(\Omega_f)),\quad D_I \in \mathcal{L}(H^s(\Gamma_I), H^{s + \frac{1}{2}}(\Omega_f)),\quad s \in \mathbb{R}.
\end{align*}

Then for given $p$ and $\vec{v}$ in appropriate spaces, we can write $p_f = \Pi_1p + \Pi_2\vec{v} + \Pi_3\vec{v}$, where
\begin{equation}\label{PiOps}
    \Pi_1p = D_I[(p)_{\Gamma_I}],\quad  \Pi_2\vec{v} = D_I\big[(\vec{e}_3 \cdot [2\nu\vec{D}(\vec{v})\vec{e}_3])_{\Gamma_I}\big], \quad \Pi_3\vec{v} = N_f[(\nu\Delta\vec{v} \cdot \vec{e}_3)_{\Gamma_f}].
\end{equation}
Finally, for $i=1,2,3$, we denote $G_i = -\nabla\Pi_i$, as defined in \eqref{A*}. With these $\Pi_i$ well-defined, we may proceed to define the semigroup generator which is shown in \cite{oldpaper} to be associated with the Biot-Stokes system \eqref{systemfull0}--\eqref{IC4*}).
The differential action of our prospective generator $\mathcal A$ is: \begin{equation}\label{diffaction}
   \mathcal A\mathbf y\equiv  \begin{pmatrix}
        \vec{w}\\
        -\mathcal{E}_0\vec{u} - \alpha \rho_b^{-1}\nabla p\\
        -\alpha c_0^{-1}\nabla\cdot \vec{w} -A_0p\\
        \rho_f^{-1}\nu \Delta \vec{v} - \rho_f^{-1}\nabla\pi(p,\bv)
    \end{pmatrix}
    =
    \begin{pmatrix}\bw_1^*\\\bw_2^*\\q^*\\\mathbf f^*\end{pmatrix} \in X, \quad \vec{y} = [\vec{u},\vec{w},p,\vec{v}]^T \in \mathcal{D}(\mathcal{A}),
\end{equation}
where we denote $\pi = \pi(p,\vec{v})\equiv \Pi_1p+\Pi_2 \mathbf v+\Pi_3\mathbf v \in L^2(\Omega_f)$ as a solution to the following elliptic problem, for given $p$ and $\vec{v}$ emanating from $\vec{y} \in \mathcal{D}(\mathcal{A})$ (i.e., having adequate regularity to form the RHS below):
\begin{equation}\label{pi}
    \begin{cases}
        \Delta \pi = 0 &\in L^2(\Omega_f),\\
        \partial_{\vec{e}_3}\pi = \nu \Delta \vec{v} \cdot\vec{e}_3& \in H^{-3/2}(\Gamma_f),\\
        \pi = p + 2\nu\vec{e}_3\cdot \vec{D}(\vec{v})\vec{e}_3  &\in H^{-1/2}(\Gamma_I),\\
        \pi  \in H^{-1/2}_{\#}(\partial \Omega_f).
    \end{cases}
\end{equation}

From the membership $\mathcal A\mathbf y \in X$, we  then read off:  $$\vec{w} = \bw_1^\ast \in \vec{U} \subset \vec{H}_\#^1(\Omega_b,$$ 
and the elliptic system
$$\begin{cases}
        -\mathcal{E}_0\vec{u} - \alpha\nabla p = \bw_2^\ast & \in \vec{L}^2(\Omega_b),\\
        -A_0p = q^\ast + \alpha\nabla\cdot{\bw_1^*} &\in L^2(\Omega_b),\\
        \rho_f^{-1}[\nu \Delta \vec{v} - \nabla\pi] = \mathbf f^\ast &\in \vec{V},\\
        \nabla\cdot\vec{v} = 0,\\
     p\big|_{\Gamma_b} = 0,\quad \vec{v}|_{\Gamma_f} = 0, \quad \vec{u}|_{\Gamma_b} = 0\\
        \vec{v}\cdot\vec{e}_3 +  k\partial_{\vec{e}_3} p = \vec{w}\cdot\vec{e}_3 &\in H^{-1/2}(\Gamma_I),\\
        \beta\vec{v}\cdot\btau +  \btau \cdot [2\nu\vec{D}(\vec{v})-\pi]\vec{e}_3 =  \beta \vec{w}\cdot\btau&\in H^{-1/2}(\Gamma_I),\\
        \sigma_b\vec{e}_3 = [2\nu\mathbf D(\bv)-\pi]\mathbf e_3 & \in \vec{H}^{-1/2}(\Gamma_I),\\
        p = - 2\nu\vec{e}_3\cdot\mathbf D(\bv)\mathbf e_3+\pi &\in H^{-1/2}(\Gamma_I).
    \end{cases} $$

This provides a definition of the domain for $\mathcal A$. All restrictions below are interpreted as traces, in the generalized sense \cite{lionsmag}, where appropriate.
\begin{definition}[Domain of $\cA$] \label{diffdomain} Let $\vec{y}\in X$. Then $\vec{y} = [\vec{u},\vec{w},p,\vec{v}]^T \in \mathcal{D}(\mathcal{A})$ if and only if the following bullets hold:
\begin{itemize}
    \item $\vec{u} \in \vec{U}$  with  $\mathcal{E}_0(\vec{u}) \in \vec{L}^2(\Omega_b)$ (so that $[{\sigma_b(\vec{u})\mathbf n}]\big|_{\partial\Omega_b} \in \vec{H}^{-1/2}(\partial \Omega_b)$);
    \item $\vec{w} \in \vec{U}$;
    
    \item $p \in  H_\#^1(\Omega_b)$ with $A_0p \in L^2(\Omega_b)$ (so that $\left.\partial_{\mathbf n}p\right|_{\partial\Omega_b} \in {H}^{-1/2}(\partial \Omega_b)$);
    
        \item $[{\sigma_b(\vec{u})\mathbf n}] \in \vec{H}_{\#}^{-1/2}(\partial \Omega_b)$ (then  
       $\sigma^E(\bu)\mathbf n\cdot \mathbf n \in H^{-1/2}_{\#}(\partial \Omega_b)$);
       \item  $\partial_{\mathbf n} p \in H^{-1/2}_{\#}(\partial \Omega_b)$
    
    \item $\vec{v} \in \vec{H}_\#^1(\Omega_f) \cap \vec{V}$ with $\vec{v}|_{\Gamma_f} = \vec{0}$;
    \item There exists $\pi \in L^2(\Omega_f)$ such that $$\nu\Delta \vec{v} - \nabla\pi \in \vec{V},$$ where $\pi=\pi(p,\bv)$ is characterized via the BVP \eqref{pi} (and so
 $\left.\sigma_f\mathbf n\right|_{\partial \Omega_f}\in \vec{H}^{-\frac{1}{2}}(\partial\Omega_f)$ and  $\left.\pi\right|_{\partial \Omega_f} \in H^{-\frac{1}{2}}(\partial\Omega_f)$;
   
    \item $2\nu \mathbf D(\bv) \big|_{\partial \Omega_f} \in \mathbf H^{-1/2}_{\#}(\partial \Omega_f)$ and ~$\pi \big|_{\partial\Omega_f} \in H^{-1/2}_{\#}(\partial \Omega_f)$;
    
    \item  $\left.\Delta \vec{v}\cdot\mathbf n\right|_{\partial \Omega_f} \in H^{-3/2}(\partial\Omega_f)$);
    
    \item  $\restri{(\vec{v} - \vec{w})\cdot\vec{e}_3} = \restri{-k\nabla p\cdot\vec{e}_3} \in H^{-1/2}(\Gamma_I)$;
    
    \item $\restri{\beta(\vec{v}-\vec{w})\cdot\btau} = \restri{-\btau\cdot\sigma_f\vec{e}_3} \in H^{-1/2}(\Gamma_I)$;
    
    \item $\restri{-\vec{e}_3 \cdot \sigma_f\vec{e}_3} = \restri{p} \in H^{-1/2}(\Gamma_I)$;

    \item $\restri{\sigma_b\vec{e}_3} = \restri{\sigma_f\vec{e}_3} \in \vec{H}^{-1/2}(\Gamma_I)$.

\end{itemize}
\end{definition}
\noindent The above terms with  traces in  Sobolev spaces of negative indices are well-defined (as extensions) when $\vec{y} \in \mathcal{D}(\cA)$.

\subsection{Statement of Previous Results \cite{oldpaper}}

Our central supporting theorem is that of semigroup generation for  $\cA$. 
\begin{theorem}
\label{th:main1}
    The operator $\mathcal A$ on $X$,defined by \eqref{diffaction}, with domain $\mathcal D(\mathcal A)$ given in Definition \ref{diffdomain}, is the generator of a strongly continuous semigroup $\{e^{\cA t}: t\geq 0\}$ of contractions on $X$. Thus, for $\mathbf y_0 \in \mathcal D(\mathcal A)$, we have $e^{\mathcal A\cdot}\mathbf y_0 \in C([0,T];\mathcal D(\mathcal A))\cap C^1((0,T);X)$ satisfying \eqref{cauchy} in the {\em strong sense} with $\mathcal F=[\mathbf 0,\mathbf 0, 0,\mathbf 0]^T$; similarly, for $\mathbf y_0 \in X$, we have $e^{\mathcal A\cdot}\mathbf y_0 \in C([0,T];X)$ satisfying \eqref{cauchy} in the {\em generalized sense} with $\mathcal F=[\mathbf 0, \mathbf 0, 0,\mathbf 0]^T$.
\end{theorem}

As a corollary, we  obtain strong and mild (or generalized) solutions to \eqref{systemfull}--\eqref{IC4*}, including  in the non-homogeneous case with generic forces. By strong solutions to  \eqref{systemfull}--\eqref{IC4*} we mean weak solutions, as in Definition \ref{weaksols}, with enough regularity that the equations in \eqref{systemfull}--\eqref{IC4*} also hold  point-wise in space-time. Through limiting procedures---using smooth data for the semigroup when $c_0>0$, and subsequently using a ``vanishing compressibility" argument as $c_0 \searrow 0$---we can obtain weak solutions for all $c_0 \ge 0$. This  is the main result of \cite{oldpaper}: weak existence of \eqref{systemfull}--\eqref{IC4*}, independent of the storage parameter $c_0\ge 0$.
\begin{theorem}\label{th:main2} Suppose $\mathbf y_0 \in X$ and $\mathbf F_f \in L^2(0,T;\vec{H}^1_{\#,\ast}(\Omega_f)\cap \vec{V})'$, $\mathbf F_b \in L^2(0,T;\mathbf L^2(\Omega_b))$, and $S \in L^2(0,T; [H^{1}_{\#,*}(\Omega_b)]')$.
   Then for any $c_0\ge 0$, the dynamics admit weak solutions, as in Definition \ref{weaksols}. A constructed weak solution satisfies the energy inequality in \eqref{eidenttt}. 
\end{theorem}

\begin{remark}
The above theorem provides existence of weak solutions---in the sense of Definition \ref{weaksols}---for all values of $c_0\ge 0$. In an indirect way, we can provide an alternative construction for $c_0>0$ as follows. We know that $\mathcal A$ is a generator of a $C_0$-semigroup on $X$. We provide a characterization of the adjoint operator $\mathcal A^*$ on $\mathcal D(\mathcal A^*)$ in the Appendix, so Ball-type weak solutions, as in \cite{ball}, follow immediately. However, in Section \ref{ballproof}, we show that weak solutions in the sense of Definition \ref{weaksols} are can also be identified in some sense with semigroup solutions. Thus, in this case, all notions of weak solution coincide, so the abstract notion of semigroup weak solution in \cite{ball} obtains the same weak solution as those constructed in \cite{oldpaper}, but only after the calculations presented in the treatment at hand.
\end{remark}

\section{Main Result and Proof}
\begin{theorem}\label{th:main} Suppose $\mathbf y_0 \in X$ and $\mathbf F_f \in L^2(0,T;\vec{H}^1_{\#,\ast}(\Omega_f)\cap \vec{V})'$, $\mathbf F_b \in L^2(0,T;\mathbf L^2(\Omega_b))$, and $S \in L^2(0,T; [H^{1}_{\#,*}(\Omega_b)]')$.
   Then for any $c_0\ge 0$, the weak solution described in Theorem \ref{th:main1} is unique. \end{theorem}
   Since the previously constructed solution satisfies the energy inequality in \eqref{eidenttt} and the problem is linear, continuous dependence in the sense of the \eqref{eidenttt} follows immediately from uniqueness. This provides as a final corollary the main result in this paper.
   \begin{corollary}
   Let $c_0 \ge 0$. Consider the initial boundary value problem in \eqref{systemfull}--\eqref{IC4*}, with data $\mathbf [\bu_0,\bu_1,d_0,\bv_0]^T \in X$,~ $\mathbf F_f \in L^2(0,T;\vec{H}^1_{\#,\ast}(\Omega_f)\cap \vec{V})'$, $\mathbf F_b \in L^2(0,T;\mathbf L^2(\Omega_b))$, and $S \in L^2(0,T; [H^{1}_{\#,*}(\Omega_b)]')$. This problem is {\em weakly well-posed} in the sense of Definition \ref{weaksols}, with continuous dependence holding in the sense of \eqref{eidenttt}. 
   \end{corollary}

\subsection{Proof for $c_0=0$ Case: Hyperbolic Regularization}\label{proof1}

Let us now prove that weak solutions are unique. For the proof, we appeal only to Definition \ref{weaksols}, and only to the regularity (and estimates) coming from the solution $[\bu,p,\bv]^T \in \mathcal V_{\text{sol}}$ and the test functions $[\bxi,q,\bzeta]^T \in \mathcal V_{\text{test}}$. We note that, as is often the case in hyperbolic problems, the solution does not possess enough regularity to be used as a test function. Moreover, the traditional approach used for hyperbolic weak solutions is convoluted by the coupled nature of the problem.

\begin{proof}[Proof of Theorem \ref{th:main} with $c_0=0$]
Suppose that we have two weak solutions in the sense of Definition \ref{weaksols}. We consider the difference of weak forms in \eqref{weakform} and, as the problem as linear, we obtain zero data; upon renaming $[\bu,p,\bv]^T \in \mathcal V_{\text{sol}}$ as the difference of two solutions, we obtain the following weak form:
    \begin{align}\label{weakformforu}
        &- \rho_b((\vec{u}_t,\bxi_t))_{\Omega_b} + ((\sigma_b(\vec{u},p), \nabla \bxi))_{\Omega_b} -\alpha ((\nabla\cdot\vec{u}, \partial_t q))_{\Omega_b} + ((k\nabla p,\nabla q))_{\Omega_b} \nn\\
        &- \rho_f((\vec{v},\bzeta_t))_{\Omega_f} + 2\nu((\vec{D}(\vec{v}),\vec{D}(\bzeta)))_{\Omega_f} + (({p,(\bzeta-\bxi)\cdot\vec{e}_3}))_{\Gamma_I} - (({\vec{v}\cdot\vec{e}_3,q}))_{\Gamma_I} \nn\\
        &- (({\vec{u}\cdot\vec{e}_3,\partial_tq}))_{\Gamma_I} + \beta (({\vec{v}\cdot\btau,(\bzeta - \bxi)\cdot\btau}))_{\Gamma_I} + \beta(({\vec{u}\cdot\btau, (\bzeta_t - \bxi_t)\cdot\btau}))_{\Gamma_I}\nn \\
        &= 0,
    \end{align}
    for every test function $[\bxi,q,\bzeta]^T \in \mathcal{V}_\text{test}$. In addition, by definition of weak solution, we may utilize bounds on $(0,T)$ for the following norms:
    \begin{align}\label{normquote}
 \vec{u} \in L^\infty(0,T;\vec{U}),~~ \vec{u}_t \in L^{\infty}(0,T;\vec{L}^2(\Omega_b)), ~~\gamma_0[\bu_t] \in L^2(0,T;L^2(\Gamma_I)),\\ \nn
  p \in L^2(0,T; H^1_{\#,*}(\Omega_b)), ~~c_0^{1/2}p \in L^{\infty}(0,T; L^2(\Omega_b)), ~~
 \bv \in  L^\infty(0,T;\vec{V}), \\ \nn \bv\in L^2(0,T; \vec{H}^1_{\#,\ast}(\Omega_f) \cap \vec{V}).
\end{align}
Of course, we may apply the trace theorem (spatially) to all functions above that have $H^1(\Omega_i)$ regularity.
    
  With zero data, we aim to interpret certain time derivatives distributionally, yielding  additional regularity in certain dual senses. Subsequently, we will be able to reinterpret several terms in the weak form \eqref{weakformforu} which will then allow testing with the solution.

First, taking $\bzeta=\mathbf 0$ and $q=0$ in \eqref{weakformforu}:
\begin{align*}
    0=- \rho_b(( {\vec{u}}_t ,\bxi_t))_{\Omega_b} +  ((\sigma_b( {\vec{u}} ,  p ), \nabla \bxi))_{\Omega_b} 
  -  (({  p ,\bxi\cdot\vec{e}_3}))_{\Gamma_I} 
-\beta \left( \mathbf{v}\cdot \tau ,\xi \cdot \tau \right) _{\Gamma
_{I}}-\beta \left( \mathbf{u}\cdot \tau ,\xi _{t}\cdot \tau \right) _{\Gamma
_{I}}.
\end{align*}
Subsequently  integrating by parts and noting that $\mathbf{u}(t=0)= \mathbf{0}$ and $\bxi \in C_0^1([0,T);\mathbf U)$, we have
\begin{align} \label{u_mid}
    -~ \rho_b(( {\vec{u}}_t ,\bxi_t))_{\Omega_b} =&~- ((\sigma_b( {\vec{u}} ,  p ), \nabla \bxi))_{\Omega_b} 
  + (({  p ,\bxi\cdot\vec{e}_3}))_{\Gamma_I} - \beta (([ {{\bu}_t-\vec{v}] \cdot\btau,\bxi\cdot\btau}))_{\Gamma_I}.
\end{align}
Finally, restricting to $\bxi \in C_0^{\infty}(0,T; \mathbf U)$, we obtain the estimate: \begin{equation*} 
\left|(( {\bu}_t,\bxi_t))_{\Omega_b}\right| \lesssim C\left(|| {\bu}||_{L^2(0,T;\mathbf U)}+||  p||_{L^2(0,T;\mathbf H^1_{\#,*}(\Omega_b)}+|| ({\bu}_t- {\bv}) \cdot\btau ||_{L^2(0,T;L^2(\Gamma_I))}\right)||\bxi||_{L^2(0,T;\mathbf U)},
\end{equation*}
where we have bounded all RHS terms above via Cauchy-Schwarz, the trace theorem,  Poincar\'e's inequality, and assumed bounds on the weak solution $[\bu,p,\bv]^T $. From this estimate, we then infer that Biot component $\bu$ of weak solution $[  \bu,   p,  {\bv}] \in \mathcal V_{\text{sol}}$ of \eqref{weakformforu} (and \eqref{u_mid}) satisfies  
\begin{equation} \label{u_tt}
 {\bu}_{tt} \in L^2(0,T; \mathbf U'),
\end{equation}
with associated bound. (In making this deduction, we are implicitly considering, from the definition of the weak solution, that $\bu_t\cdot \btau \in L^2(0,T; L^2(\Gamma_I))$).

In addition, with given test function $\bxi$ in $C_0^{\infty}((0,T)\times \Omega_b)$, we infer from \eqref{u_mid} the relation,
\begin{align*}
    ~(\langle  \rho_b {\vec{u}}_{tt},\bxi \rangle)_{\Omega_b}   -  (\langle\nabla \cdot \sigma^E(\bu) ,\bxi \rangle)_{\Omega_b}  
  = - \alpha((\nabla p, \bxi))_{\Omega_b}.
\end{align*}
Since $\nabla p \in L^2(0,T;\mathbf L^2(\Omega_b))$ for a weak solution,  we then have

\begin{equation} \label{higher}
\rho_b {\vec{u}}_{tt} -  \nabla \cdot \sigma^E(\bu) = -\alpha \nabla p \in L^2(0,T;\mathbf L^2(\Omega_b)).
\end{equation}

In turn,  \eqref{u_tt} and \eqref{higher} give

\begin{equation} \label{second}
  \nabla \cdot \sigma^E(\bu) \in L^2(0,T; \mathbf U').
\end{equation}

\medskip

 Subsequently, we can give a meaning to the flux term $\left[  \sigma ^{E}(\mathbf{\bu })\mathbf e_3\right] _{\Gamma _{I}}$. To this end, we invoke the trace space  of $\mathbf U = \mathbf H^1_{\#,*}(\Omega_b)$; namely, the image $\gamma _{0}[ \mathbf U] $, a closed subspace of $\mathbf H^{\frac{1}{2}}(\partial \Omega _{b})$. We denote its dual by $\gamma_0[\mathbf U]'$.
With this trace space in mind, we consider the continuous right inverse $$
\gamma _{0}^{+}\in \mathscr{L}(\mathbf H^{\frac{1}{2}}(\partial \Omega
_{b}),\mathbf H^{1}(\Omega _{b}));$$ viz., $\gamma _{0} \gamma _{0}^{+}(\mathbf g)= \mathbf g$, for $
\mathbf g\in \mathbf H^{\frac{1}{2}}(\partial \Omega _{b})$. (See 
\cite[Theorem 3.38, p.102]{mclean}.)
Therewith, let $\mathbf{g}\in L^{2}(0,T;\gamma _{0}[\mathbf U])$ be given. Then,%

\[
\left( \left\langle  \sigma ^{E}(\mathbf{u})\mathbf{e}_{3},\mathbf{g}\right\rangle
\right) _{\Gamma _{I}}=-\left( \left( \nabla \cdot \sigma ^{E}(\mathbf{u}%
),\gamma _{0}^{+}(\mathbf{g})\right) \right) _{\Omega _{b}}-\left( \left(
\sigma ^{E}(\mathbf{u}),\mathbf{D}(\gamma _{0}^{+}(\mathbf{g}))\right)
\right) _{\Omega _{b}}. \label{Lax}
\]%

An estimation of right hand side above, in part via \eqref{second}, thus yields,
\begin{equation}
\sigma ^{E}(\mathbf{u}) \mathbf{e}_{3} \in L^{2}(0,T;[\gamma _{0}(\mathbf U)]^{\prime }).  \label{Lax2}
\end{equation}

\medskip

Now, in \eqref{u_mid}, we take $\bxi \in C_0^{\infty}(0,T; \mathbf U)$. Applying \eqref{higher} and \eqref{Lax2} thereto, we may then write:

\begin{equation}
\left( \left\langle  \sigma ^{E}(\mathbf{u})\mathbf{e}_{3} ,\bxi%
\right\rangle \right) _{\Gamma _{I}}-\alpha \left( \left( p,\bxi\cdot 
\mathbf{e}_{3}\right) \right) _{\Gamma _{I}}+\left( \left( p,\bxi\cdot 
\mathbf{e}_{3}\right) \right) _{\Gamma _{I}}-\beta \left( \left( \lbrack 
\mathbf{u}_{t}-\mathbf{v}]\cdot \mathbf{\btau},\bxi\cdot \mathbf{\btau}%
\right) \right) _{\Gamma _{I}}=0.  \label{3.85}
\end{equation}

Or, upon rearrangement, for $\bxi \in C_0^{\infty}(0,T; \mathbf U)$,
\begin{equation}
\begin{array}{l}
(\alpha -1)\left( \left( p,\bxi \cdot \mathbf{e}_{3}\right) \right)
_{\Gamma _{I}}=\left( \left\langle  \sigma^E (\mathbf{u})\mathbf{e}_{3},
\bxi \right\rangle \right) _{\Gamma _{I}} -\beta \left( \left( \lbrack \mathbf{u}_{t}-\mathbf{v}%
]\cdot \btau ,\bxi \cdot \btau \right) \right)
_{\Gamma _{I}}.%
\end{array}
\label{tousemin}
\end{equation}
This identity will be crucial later on.

We will now focus on the Biot pressure equation. First, since we have that $\bu \in L^2(0,T;\mathbf U)$, we have that $\gamma_0[\bu] \in \mathbf L^2(0,T;\mathbf{H}^{1/2}(\Gamma_I))$ by the Sobolev trace theorem (restricting to $\Gamma_I$). Moreover, the fluid content $\nabla\cdot \bu \in L^2(0,T;L^2(\Omega_b))$, by the definition of weak solution. Hence the time derivatives exist in the sense of distributions: $\gamma_0[\bu]_t \in \mathscr D'((0,T) \times \Gamma_I)$ and $\nabla \cdot \bu_t \in \mathscr D'((0,T) \times \Omega_b)$. We can in fact say more about the regularity of these terms.  Invoking the test functions $\bzeta=\bxi=\mathbf 0$ in \eqref{weakformforu}, we obtain
\begin{align*}
   - \alpha ((\nabla\cdot {\vec{u}} , \partial_t q))_{\Omega_b} & 
    - (( {\vec{u}} \cdot\vec{e}_3,\partial_tq))_{\Gamma_I} =   - ((k\nabla   p ,\nabla q))_{\Omega_b} +  (({ {\vec{v}} \cdot\vec{e}_3,q}))_{\Gamma_I} .
\end{align*}    
Restricting to $q \in C_0^{\infty}((0,T); H^1_{\#,*}(\Omega_b))$, we then have
\begin{align*}
    -\alpha (( {\vec{u}}_t, \nabla  q))_{\Omega_b} &+ (\alpha-1) ((\bu\cdot \mathbf e_3, q_t))_{\Gamma_I}
   =   - ((k\nabla   p ,\nabla q))_{\Omega_b} +  (({ {\vec{v}} \cdot\vec{e}_3,q}))_{\Gamma_I}.
\end{align*}    
Thus, the action of the distribution $[\bu_t\cdot \mathbf e_3]$
 is given through the identity
\begin{align}\label{onetouse***}
    (\alpha-1) ((\bu\cdot \mathbf e_3, q_t))_{\Gamma_I}
   =  \alpha (( {\vec{u}}_t, \nabla  q))_{\Omega_b}   - ((k\nabla   p ,\nabla q))_{\Omega_b} +  (({ {\vec{v}} \cdot\vec{e}_3,q}))_{\Gamma_I},\\ \nonumber~~\forall q \in C_0^{\infty}((0,T);H^1_{\#,*}(\Omega_b)).
\end{align}    

 Via \eqref{onetouse***},  we can now characterize the regularity of the distribution $\mathbf u_t\cdot \mathbf e_3$ defined on $\Gamma$. We re-invoke the trace space  $\gamma _{0}[ H_{\#,\ast }^{1}(\Omega _{b}))] $ and continuous right inverse $%
\gamma _{0}^{+}\in \mathscr{L}(H^{\frac{1}{2}}(\partial \Omega
_{b}),H^{1}(\Omega _{b}))$. Taking test function $q\equiv \gamma _{0}^{+}(g)$ in (%
\ref{onetouse***}), where $g\in C_0^{\infty}((0,T);\gamma _{0}[H_{\#,\ast }^{1}(\Omega _{b}])$ is arbitrary, we then have
\begin{equation}
(\alpha -1)\left( \left( \mathbf{u}\cdot \mathbf{e}_{3},g_{t}\right) \right)
_{\Gamma _{I}}=\alpha \left( \left( \mathbf{u}_{t},\nabla \gamma
_{0}^{+}(g)\right) \right) _{\Omega _{b}}-\left( \left( k \nabla
p,\nabla \gamma _{0}^{+}(g)\right) \right) _{\Omega _{b}}+\left( \left( 
\mathbf{v}\cdot \mathbf{e}_{3},g\right) \right) _{\Gamma _{I}}\text{. } \label{extra}
\end{equation}
By the density of $C_0^{\infty}((0,T);\gamma _{0}[H_{\#,\ast }^{1}(\Omega _{b}])$ in $%
L^{2}(0,T;\gamma _{0}[H_{\#,\ast }^{\frac{1}{2}}(\Omega _{b}])$, this relation holds for all \newline
$g$ in $L^2(0,T;\gamma _{0}[H_{\#,\ast }^{\frac{1}{2}}(\Omega
_{b})])$. An estimation of the right hand side of (\ref{extra}) thus yields the conclusion that \begin{equation}\label{traceone}\bu_t\cdot \mathbf e_3 \in L^2(0,T; [\gamma_0[H^1_{\#,*}(\Omega_b)]').\end{equation}

As a consequence of this regularity, as well as the density of $C_0^{\infty}((0,T);H^1_{\#,*}(\Omega_b))$ in $L^2(0,T;H^1_{\#,*}(\Omega_b))$, the relation \refeq{onetouse***} holds for any test function $q \in L^2(0,T;H^1_{\#,*}(\Omega_b))$. Taking now $q= \phi p$, where $p$ is the Biot pressure weak solution and $\phi \in C_0^{\infty}(0,T)$ arbitrary, we then have, 
\begin{align}\label{toplugin}
   (1-\alpha)(\langle \bu_t\cdot \mathbf e_3, p\rangle, \phi )_{\Gamma_I}
   =  \alpha (( {\vec{u}}_t, \nabla  p), \phi )_{\Omega_b}   - ((k\nabla   p ,\nabla p), \phi )_{\Omega_b} +  (({ {\vec{v}} \cdot\vec{e}_3,p}),\phi )_{\Gamma_I}.
\end{align}

We invoke the relation \eqref{tousemin}, with test function $\bxi$ replated by $\phi \bxi$, and $\phi$ as above;  adding that to \eqref{toplugin}, and also adding by zero via \eqref{higher}, we obtain the variational relation:
\begin{align}\nonumber
\left( \left\langle  \sigma ^{E}(\mathbf{u})\mathbf{e}_{3},\bxi%
\right\rangle, \phi \right) _{\Gamma _{I}}= & ~\left( \left( \rho _{b}\mathbf{u}%
_{tt}-\nabla \cdot \sigma ^{E}(\mathbf{u}),\bxi\right), \phi \right) _{\Omega
_{b}} +
(1-\alpha )\left( \left\langle [\mathbf{u}_{t}-\bxi]\cdot \mathbf{e}%
_{3},p\right\rangle, \phi \right) _{\Gamma _{I}} \\ \label{semi} &+\beta \left( \left( \lbrack 
\mathbf{u}_{t} -\mathbf{v}]\cdot \mathbf{\btau},\bxi\cdot \mathbf{\btau}%
\right), \phi \right) _{\Gamma _{I}} 
+\left( \left( k \nabla p,\nabla p\right), \phi \right) _{\Omega
_{b}}-\left( \left( \mathbf{v}\cdot \mathbf{e}_{3},p\right), \phi \right)
_{\Gamma _{I}} \\ & \text{ \ \ for }\bxi\in C_0^{\infty}((0,T);\mathbf{H}_{\#,\ast
}^{1}(\Omega _{b}))\text{.} \nonumber
\end{align}

Now let $\left\{ \bxi_{n}\right\} \subset C_0^{\infty}((0,T);\mathbf{H}_{\#,\ast
}^{1}(\Omega _{b}))$ constitute an approximation of $\mathbf{u}_{t}$ in the sense of $L^2(0,T;\mathbf L^2(\Omega_b))$.
Setting $\bxi\equiv \bxi_{n}$ in (\ref{semi}) and passing to the limit, we
then obtain, for arbitrary $\phi \in C_0^{\infty}(0,T)$,
\begin{align}\nonumber
\left( \left\langle \sigma ^{E}(\mathbf{u}) \mathbf{e}_{3},\mathbf{u}%
_{t}\right\rangle, \phi \right) _{\Gamma _{I}}=&~ \left( \left( \rho _{b}\mathbf{%
u}_{tt}-\nabla \cdot \sigma ^{E}(\mathbf{u}),\mathbf{u}_{t}\right), \phi
\right) _{\Omega _{b}} + \beta \left( \left( \lbrack \mathbf{u}_{t}-\mathbf{v}%
]\cdot \mathbf{\btau},\mathbf{u}_{t}\cdot \mathbf{\btau}\right), \phi \right)
_{\Gamma _{I}} \\ 
& +\left( \left( k \nabla p,\nabla p\right), \phi \right) _{\Omega
_{b}}-\left( \left( \mathbf{v}\cdot \mathbf{e}_{3},p\right), \phi \right)
_{\Gamma _{I}}\text{.}%
\label{semi_2}
\end{align}
Considering the right hand side of the above relation, the definition of weak
solutions in Definition \ref{weaksols}, and \eqref{higher} we can infer the well-definition of the function%
\begin{equation}
t\mapsto \left\langle  \sigma ^{E}(\mathbf{u}(t))\mathbf{e}_{3},%
\mathbf{u}_{t}(t)\right\rangle \in L^{1}(0,T).
\label{semi_3}
\end{equation}

At this point, we collect the following facts:
\begin{itemize}
\item[(s.i)]  $\bu \in L^2(0,T;\mathbf U)$;~ $\bu_t \in L^2(0,T;\mathbf L^2(\Omega_b))$
\item[(s.ii)] $\rho_b\bu_{tt} -\nabla \cdot \sigma^E(\bu)  =-\alpha \nabla p   \in L^2(0,T;\mathbf L^2(\Omega_b))$
\item[(s.iii)] $\left\langle  \sigma ^{E}(\mathbf{u})\mathbf{e}_{3},\mathbf{u}%
_{t}\right\rangle \in L^{1}(0,T)$ (as noted in \ref{semi_3})).
\end{itemize}
In consequence of the above items, we proceed analogously as in \cite{temam} for second order (in time) evolutions---see the proof of \cite[Theorem 4.1, p.77]{temam1}. Given (s.i)--(s.iii), we will have that, after an readjustment $a.e.$ in time, the displacement component of the weak solution, $\bu$, satisfies the following (invoking \cite[Lemma 4.1, p.78]{temam1}):
\begin{itemize}
\item  $\bu \in C([0,T];\mathbf U) \cap C^1([0,T];\mathbf L^2(\Omega_b))$ 
\item the  identity  (where the time derivative on the RHS is interpreted in the sense of $\ds \mathscr D'(0,T)$): \begin{equation} \label{A}\ds \big(\rho_b\bu_{tt} -\nabla \cdot \sigma^E(\bu),\mathbf u_t\big)_{\mathbf L^2(\Omega_b)} - \left\langle  \sigma ^{E}(\mathbf{u})\mathbf{e}_{3},\mathbf{u}%
_{t}\right\rangle= \dfrac{1}{2}\dfrac{d}{dt}\left[\rho_b||\bu_t||_{\mathbf L^2(\Omega_b)}^2+||\bu||_E^2\right].\end{equation} 
\end{itemize}
We may now use the identity \eqref{A}. 

Since test function $\phi$ in \eqref{semi} is arbitrary, we obtain the following relation for the weak solution:
\begin{align}
\left(\left\langle  \sigma ^{E}(\mathbf{u})\mathbf{e}_{3},\mathbf{u}%
_{t}\right\rangle  \right) _{\Gamma _{I}}=&~ \left( \left( \rho _{b}\mathbf{%
u}_{tt}-\nabla \cdot \sigma ^{E}(\mathbf{u}),\mathbf{u}_{t}\right) 
\right) _{\Omega _{b}} + \beta \left( \left( \lbrack \mathbf{u}_{t}-\mathbf{v}%
]\cdot \mathbf{\btau},\mathbf{u}_{t}\cdot \mathbf{\btau}\right)  \right)
_{\Gamma _{I}} \\ 
&+\left( \left( k \nabla p,\nabla p\right) \phi \right) _{\Omega
_{b}}-\left( \left( \mathbf{v}\cdot \mathbf{e}_{3},p\right)  \right)
_{\Gamma _{I}}\text{.}%
\label{semi_5}
\end{align}
Invoking \eqref{A} gives now,
\begin{equation}\label{dualone}
 - ((k\nabla   p ,\nabla p))_{\Omega_b} +  (({ {\vec{v}} \cdot\vec{e}_3,p}))_{\Gamma_I}= \frac{\rho_b}{2}||\bu_t(T)||_{\Omega_b}^2+\frac{1}{2}||\bu(T)||_E^2+\beta(([\bu_t-\bv]\cdot \btau,\bu_t\cdot \btau)), 
\end{equation}

Finally, we repeat our  procedure (in a more standard way) for the Stokes velocity variable in the weak form. We take $\bxi=\mathbf 0$ and $q=0$ in \eqref{weakformforu} to obtain:
\begin{align}\label{onetouse}
    - \rho_f(( {\bv} ,\bzeta_t))_{\Omega_f} +&~ 2\nu((\vec{D}( {\vec{v}} ),\vec{D}(\bzeta)))_{\Omega_f} + (({  p ,\bzeta\cdot\vec{e}_3}))_{\Gamma_I}+ \beta (({[ {\vec{v}}- {\bu}_t] \cdot\btau,\bzeta \cdot\btau}))_{\Gamma_I}=0.
\end{align}

We may invoke the trace theorem for $\bv\in L^2(0,T; \vec{H}^1_{\#,\ast}(\Omega_f) \cap \vec{V})$ and $p \in L^2(0,T; H^1_{\#,*}(\Omega_b))$, as well as the regularity for $\bu_t\cdot \btau \in L^2(0,T;L^2(\Gamma_I))$ to estimate the term $ \rho_f(( {\bv} ,\bzeta_t))_{\Omega_f}$. Noting that $\bzeta \in C_0^1([0,T); \mathbf H^1_{\#,*}(\Omega_f) \cap \bV)$, we obtain immediately from \eqref{onetouse} that distributional derivative
\begin{equation}
 {\bv}_t \in L^2(0,T; [\mathbf H^1_{\#,*}(\Omega_f) \cap \bV]').
\end{equation}
From which we have, by the Intermediate Derivatives Theorem, that $ \mathbf{v} \in C([0,T]; \mathbf{L}^2(\Omega_f))$ and, as is standard \cite{evans} (in the sense of $\mathscr D'(0,T)$),
\begin{equation}
\dfrac{1}{2}\dfrac{d}{dt}||\bv||^2_{\mathbf L^2(\Omega_f)} = \langle \bv_t,\bv\rangle,
\end{equation}
where the RHS represents the appropriate duality pairing,
taking into account $\mathbf H^1_{\#,*}(\Omega_f) \cap \bV \hookrightarrow \mathbf L^2(\Omega_f) \hookrightarrow [\mathbf H^1_{\#,*}(\Omega_f) \cap \bV]'$. Moreover, by density, equation \eqref{onetouse} holds with $\bzeta = \mathbf{v}$. Thus, for \eqref{onetouse} above, we can reverse integration by parts in time and test with $\bv$ to obtain the relation:
\begin{align}\label{onetouse*}
\dfrac{\rho_f}{2}||\bv(T)||^2+&~ 2\nu((\vec{D}( {\vec{v}} ),\vec{D}(\bzeta)))_{\Omega_f} + (({  p ,\bv\cdot\vec{e}_3}))_{\Gamma_I}+ \beta (({[ {\vec{v}}- {\bu}_t] \cdot\btau,\bv \cdot\btau}))_{\Gamma_I}=0.
\end{align}

Finally, we consider the \eqref{dualone} and \eqref{onetouse} in conjunction:
\begin{align}\small
\frac{\rho_b}{2}||\bu_t(T)||_{\Omega_b}^2+\frac{1}{2}||\bu(T)||_E^2+k||\nabla p||^2_{L^2(0,T;\mathbf L^2(\Omega_b))} -   (({ {\vec{v}} \cdot\vec{e}_3,p}))_{\Gamma_I} +\beta(([\bu_t-\bv]\cdot \btau,\bu_t\cdot \btau))=&~0 \\ \small
\dfrac{\rho_f}{2}||\bv(T)||^2+~ 2\nu((\vec{D}( {\vec{v}} ),\vec{D}(\bzeta)))_{\Omega_f} + (({  p ,\bv\cdot\vec{e}_3}))_{\Gamma_I}+ \beta (({[ {\vec{v}}- {\bu}_t] \cdot\btau,\bv \cdot\btau}))_{\Gamma_I}=&~0.
\end{align}
Adding the equalities, we have obtained the energy identity for this weak solution (with zero RHS and zero Cauchy data). Moreover, since we can replace $T$ in the energy identity with $0<t<T$, we conclude that 
$[\bu,\bu_t,p,\bv]^T=[\mathbf 0, \mathbf 0, 0,\mathbf 0]^T$ point-wise $a.e.$ in spacetime. From this, uniqueness of weak solutions follows for the case $c_0=0$. 

\end{proof}

\subsection{Proof for $c_0>0$ Case: Ball's Method}\label{ballproof}
In this section we will invoke an approach due to Ball \cite{ball} for the present Hilbert space setting. The abstract notion of semigroup weak solution requires test functions coming from $\mathcal D(\mathcal A^*)$, where $\mathcal A^*$ is the Hilbert space adjoint of the Biot-Stokes semigroup generator $\mathcal A \subset X \to X$ defined in \eqref{diffaction} for $c_0>0$. In this work, we will need an explicit representation of the adjoint operator $\mathcal A^*$. The computation of the adjoint, and, in particular, the characterization of its domain $\mathcal D(\mathcal A^*)$, constitutes the considerable contribution here, as the complexity of the given Biot-Stokes PDE dynamics render these highly nontrivial tasks. The entirety of the Appendix  is devoted to this endeavor. With the adjoint well in hand, we can reconcile Ball's abstract notion of weak solution---valid for any semigroup generator---with our specific notion of weak solution in Definition \ref{weaksols}. To wit, we will find that a weak solution in the sense of the latter is a weak solution in the sense of the former. Given the uniqueness of weak solutions derived from the semigroup generator, we will then establish that weak solutions in the sense of Definition \ref{weaksols} are also unique, through identification. 

We begin by stating the characterization of $\mathcal A^*$. The proof of this Proposition is relegated to the Appendix, as it is both technical and protracted. 
\begin{proposition}\label{thistheadjoint}
Consider $\mathcal A$ as in \eqref{diffaction}, with $\mathcal D(\mathcal A)$ as defined in Definition \ref{diffdomain}. Then the Hilbert space adjoint $\mathcal A^*: \mathcal D(\mathcal A^*) \subset X \to X$ is given by: 
\begin{equation}\label{diffaction*}
   \mathcal A^*\mathbf{\widetilde y}\equiv  \begin{pmatrix}
       - \widetilde{\vec{w}}\\
        \mathcal{E}_0{\widetilde{\vec{u}}} + \alpha \rho_b^{-1}\nabla \widetilde p\\
        \alpha c_0^{-1}\nabla\cdot \widetilde{\vec{w}} -A_0\widetilde p\\
        \rho_f^{-1}\nu \Delta \widetilde{\vec{v}} - \rho_f^{-1}\nabla \widetilde\pi\end{pmatrix}
    \in X, \quad \widetilde {\mathbf y} = [\widetilde{\bu}, \widetilde{\bw}, \widetilde p, \widetilde{\bv}]^T \in \mathcal D(\mathcal A^*),
\end{equation}
where we denote $\widetilde \pi = \widetilde \pi(p,\vec{v})$ (given through the Green's maps as in Section \ref{operator}) as the solution to 
\begin{equation}\label{pi*}
    \begin{cases}
        \Delta \widetilde \pi = 0 &\in L^2(\Omega_f),\\
        \partial_{\vec{e}_3}\widetilde \pi = \nu \Delta \widetilde{\vec{v}} \cdot\vec{e}_3& \in H^{-3/2}(\Gamma_f),\\
       \widetilde \pi = \mathbf e_3\cdot[\sigma_b(\widetilde{\bu},\widetilde p)] + 2\nu\vec{e}_3\cdot \vec{D}(\widetilde{\vec{v}})\vec{e}_3  &\in H^{-1/2}(\Gamma_I),
    \end{cases}
\end{equation}
with periodic Dirichlet boundary conditions---appropriately defined---on $\Gamma_{lat}$. 

The Sobolev regularity of elements in $\mathcal D(\mathcal A^*)$ is identical to those in Definition \ref{diffdomain}, and the associated boundary conditions for $[\widetilde{\bu}, \widetilde{\bw}, \widetilde p, \widetilde{\bv}]^T \in \mathcal D(\mathcal A^*)$ on $\Gamma_I$ are:
\begin{itemize}
\item $\sigma_f(\widetilde{\bv},\widetilde \pi)\mathbf e_3=-\sigma_b(\widetilde{\bu},\widetilde p)\mathbf e_3$,
\item $-[\sigma_f(\widetilde{\bv},\widetilde \pi)\mathbf e_3]\cdot \btau=\beta [\widetilde{\bv}-\widetilde{\bw}]\cdot \btau$,
\item $\widetilde p = [\sigma_f(\widetilde{\bv},\widetilde \pi)\mathbf e_3]\cdot\mathbf e_3$,
\item $k\nabla \widetilde p\cdot \mathbf e_3 = [\widetilde{\bv}-\widetilde{\bw}]\cdot \mathbf e_3$,
\end{itemize}
interpreted in the appropriate sense (as in Definition \ref{diffdomain}).
\end{proposition}

Since $\mathcal A$ generates a $C_0$-semigroup of contractions on the Hilbert space $X$ via the Lumer-Phillips Theorem, we know immediately that $\mathcal A^*$, with domain given by $\mathcal D(\mathcal A^*)$, generates a $C_0$-semigroup of contractions on $X$ as well \cite[Ch.1, Corollary 10.6]{pazy}. 
\begin{proof}[Proof of Theorem \ref{th:main} with $c_0>0$] We recall the variational relation \eqref{weakform}, for weak solutions with $c_0 > 0$, 
    \begin{align}\label{weakformforu*}
        &- \rho_b((\vec{u}_t,\bxi_t))_{\Omega_b} + ((\sigma_b(\vec{u},p), \nabla \bxi))_{\Omega_b} - ((c_0p+\alpha\nabla\cdot\vec{u}, \partial_t q))_{\Omega_b} + ((k\nabla p,\nabla q))_{\Omega_b} \nn\\
        &- \rho_f((\vec{v},\bzeta_t))_{\Omega_f} + 2\nu((\vec{D}(\vec{v}),\vec{D}(\bzeta)))_{\Omega_f} + (({p,(\bzeta-\bxi)\cdot\vec{e}_3}))_{\Gamma_I} - (({\vec{v}\cdot\vec{e}_3,q}))_{\Gamma_I} \nn\\
        &- (({\vec{u}\cdot\vec{e}_3,\partial_tq}))_{\Gamma_I} + \beta (({\vec{v}\cdot\btau,(\bzeta - \bxi)\cdot\btau}))_{\Gamma_I} + \beta(({\vec{u}\cdot\btau, (\bzeta_t - \bxi_t)\cdot\btau}))_{\Gamma_I}\nn \\
        =&~ \rho_b(\mathbf u_1,\bxi)_{\Omega_b}\big|_{t=0} + (d_0,q_b)_{\Omega_b}\big|_{t=0} + \rho_f(\vec{v}_0,\bzeta)_{\Omega_f}\big|_{t=0} - ({\vec{u}_0\cdot\vec{e}_3,q_b})_{\Gamma_I}\big|_{t=0} \nn\\
        &~+ \beta({\vec{u}_0\cdot\btau, (\bzeta - \bxi)\cdot\btau})_{\Gamma_I}\big|_{t=0} + ((\vec{F}_b,\bxi))_{\Omega_b} + (\langle S,q_b\rangle)_{\Omega_b} + (\langle \vec{F}_f,\bzeta\rangle)_{\Omega_f},
    \end{align}
       for every test function $[\bxi,q,\bzeta]^T \in \mathcal{V}_\text{test}$. Now, choosing test functions 
       \[
[\bxi,\bzeta]^T \in  \mathcal{D}((0,T); \vec{U}   \times (\vec{H}^1_{\#,\ast}(\Omega_f)\cap \vec{V}))  ~~\text{and}~q=0,
\]
and using the density thereof in $L^2(0,T;\vec{U}  \times (\vec{H}^1_{\#,\ast}(\Omega_f)\cap \vec{V}))$ -- as well as considering criterion (2) of weak solution Definition \ref{weaksols} -- we obtain from \eqref{weakformforu*} that
\begin{equation}
 {\bu}_{tt} \in L^2(0,T; \mathbf U'),~~\text{and}~~ {\bv}_t \in L^2(0,T; [\mathbf H^1_{\#,*}(\Omega_f) \cap \bV]').
\end{equation}

Thus, similar to the previous case of $c_{0}=0$, we then have from the Intermediate
Derivatives Theorem that 
\begin{equation}
\mathbf{u}\in C([0,T];\mathbf{U})\cap C^{1}([0,T];\mathbf{L}^{2}(\Omega
_{b}))\text{ and }\mathbf{v}\in C([0,T]; \mathbf L^2(\Omega_f)).  \label{high_2}
\end{equation}

However, the treatment of the terms
$$ ((c_0p+\alpha\nabla\cdot\vec{u}, \partial_t q))_{\Omega_b},~~\text{ and }~~(({\vec{u}\cdot\vec{e}_3,\partial_tq}))_{\Gamma_I} $$ is necessarily different and more subtle than the $c_0=0$ case.  
Taking test functions $q\in C_{0}^{1}([0,T);H_{\#,\ast }^{1}(\Omega _{b}))$ and  $\bzeta=\bxi=\mathbf 0$ in \eqref{weakformforu*} we obtain,
\begin{align*}
   -  ((c_0p+\alpha \nabla\cdot {\vec{u}} , \partial_t q))_{\Omega_b} & 
    - (( {\vec{u}} \cdot\vec{e}_3,\partial_tq))_{\Gamma_I} =   - ((k\nabla   p ,\nabla q))_{\Omega_b} +  (({ {\vec{v}} \cdot\vec{e}_3,q}))_{\Gamma_I}. 
\end{align*}    
Concerning the second term on left hand side: we employ the divergence theorem to rewrite it, after considering the zero Cauchy data for the Biot-Stokes solution and zero data at $t=T$ for the test function $q$:
\begin{align*}
(( {\vec{u}} \cdot\vec{e}_3,\partial_tq))_{\Gamma_I} =&~ ((\nabla \cdot \bu, q_t))_{\Omega_b}+((\bu , \nabla q_t))_{\Omega_b} \\
=&~((\nabla \cdot \bu, q_t))_{\Omega_b}-((\bu_t, \nabla q))_{\Omega_b}
\end{align*}
So the weak form for the pressure component can be rewritten as
\begin{align*}
   -  ((c_0p+(\alpha+1) \nabla\cdot {\vec{u}} , \partial_t q))_{\Omega_b} & 
=  - ((\bu_t, \nabla q))_{\Omega_b}- ((k\nabla   p ,\nabla q))_{\Omega_b} +  (({ {\vec{v}} \cdot\vec{e}_3,q}))_{\Gamma_I} 
\end{align*}    
for all $q \in C_0^{1}([0,T);H^1_{\#,*}(\Omega_b))$.
This shows, upon estimating the RHS directly via the available bounds from the definition of weak solution in \eqref{weakform}, and subsequently invoking density, that
$$\partial_t[c_0p+(\alpha+1) \nabla\cdot {\vec{u}} ] \in L^2(0,T;[H^1_{\#,*}(\Omega_b)]').$$
Moreover, since $p \in L^2(0,T;H^1_{\#,*}(\Omega_b))$ and $\nabla \cdot \bu \in C([0,T];L^2(\Omega_b))$ (from \eqref{high_2}), we again obtain from the Intermediate Derivatives Theorem that (conservatively) 
$$c_0p+(\alpha+1) \nabla\cdot {\vec{u}} \in C([0,T];[H^1_{\#,*}(\Omega_b)]'),$$
which yields
\begin{equation}
p\in C([0,T];[H_{\#,\ast }^{1}(\Omega _{b})]^{\prime }).  \label{p_r}
\end{equation}

Now, the goal is to move to a static test function version of the weak form \eqref{weakformforu*}, in order to reconcile our weak solution with that of the semigroup weak solution in \cite{ball}. To this end, we take the test function in relation \eqref{weakformforu*} to be $[\phi \bxi, \phi q, \phi \bzeta]^T$, where (static) $[\bxi, q, \bzeta]^T \in \vec{U} \times  H^1_{\#,*}(\Omega_b) \times (\vec{H}^1_{\#,\ast}(\Omega_f)\cap \vec{V})$ and $\phi \in \mathcal{D}(0,T)$. Integrating by parts and noting $\phi$ vanishes at $t=0,T$, we then have
    \begin{align}
        & \rho_b(\langle\vec{u}_{tt},\phi \bxi\rangle)_{\Omega_b} + ((\sigma^E(\vec{u}), \phi \nabla \bxi))_{\Omega_b}-\alpha((p,\phi \nabla \cdot \bxi))_{\Omega_b} + (\langle [c_0p+\alpha\nabla\cdot\vec{u}]_t,  \phi q\rangle )_{\Omega_b} + \left( \left( \mathbf{u}_{t}\cdot \mathbf{e}_{3},\phi q\right) \right)
_{\Gamma _{I}}
 \nn\\
        &+ ((k\nabla p, \phi \nabla q))_{\Omega_b} +  \rho_f(\langle \vec{v}_t,  \phi \bzeta\rangle)_{\Omega_f} + 2\nu((\vec{D}(\vec{v}),\vec{D}( \phi \bzeta)))_{\Omega_f} + (({p, \phi (\bzeta-\bxi)\cdot\vec{e}_3}))_{\Gamma_I} - (({\vec{v}\cdot\vec{e}_3, \phi q}))_{\Gamma_I} \nn\\
        & + \beta (({\vec{v}\cdot\btau, \phi (\bzeta - \bxi)\cdot\btau}))_{\Gamma_I} - \beta(( {\vec{u}_t\cdot\btau, \phi (\bzeta - \bxi)\cdot\btau}))_{\Gamma_I}\nn \\
        &= ((\vec{F}_b,\phi \bxi))_{\Omega_b} + (\langle S,\phi q\rangle)_{\Omega_b} + (\langle \vec{F}_f,\phi \bzeta\rangle)_{\Omega_f}.
    \end{align}

   That is, in $\mathscr{D}'(0,T)$, we have the following relation for every (static) $[\bxi, q, \bzeta]^T \in \vec{U} \times  H^1_{\#,*}(\Omega_b) \times (\vec{H}^1_{\#,\ast}(\Omega_f)\cap \vec{V})$: 
       \begin{align}\label{tointegrate}
        & \dfrac{d}{dt}\left\{ \rho_b(\vec{u}_{t},\bxi)_{\Omega_b}+ \rho_f( \vec{v},\bzeta)_{\Omega_f}  +( [c_0p+\alpha\nabla\cdot\vec{u}],  q )_{\Omega_b} + ({\vec{u}\cdot\vec{e}_3,q})_{\Gamma_I}\right\} \nn\\
        & + (k\nabla p,\nabla q)_{\Omega_b} + (\sigma^E(\vec{u}), \nabla \bxi)_{\Omega_b}-\alpha(p,\nabla \cdot \bxi)_{\Omega_b}  + 2\nu(\vec{D}(\vec{v}),\vec{D}(\bzeta))_{\Omega_f}  \nn\\
        &+ ({p,(\bzeta-\bxi)\cdot\vec{e}_3})_{\Gamma_I} - ({\vec{v}\cdot\vec{e}_3,q})_{\Gamma_I}+ \beta ([{\vec{v}-\bu_t]\cdot\btau,(\bzeta - \bxi)\cdot\btau})_{\Gamma_I} \nn \\
       &=(\vec{F}_b,\bxi)_{\Omega_b} + ( S,q)_{\Omega_b} + ( \vec{F}_f, \bzeta)_{\Omega_f} .
    \end{align}
We note that each term in this relation \eqref{tointegrate} is well-defined as a distribution,
given the regularity of weak solutions, as specified in Definition \ref{weaksols}. In particular, 
\begin{equation}
\mathbf{u}_{t}\in \left[ \mathcal{D}(0,T;H^{\frac{1}{2}}(\Gamma _{I})\right]
^{\prime }\text{ }\implies \text{ }\left( \mathbf{u}_{t}\cdot \mathbf{e}%
_{3},g\right) _{\Gamma _{I}}\in \mathcal{D}^{\prime }(0,T)\text{ for any }%
g\in H^{\frac{1}{2}}(\Gamma _{I}).  \label{trace}
\end{equation}

\smallskip

 With respect to the relation \eqref{tointegrate}, we now specify the test function to be $[\widetilde{\bu},\widetilde{\bw},\widetilde p, \widetilde{\bv}]^T \in \mathcal D(\mathcal A^*)$. (Since $[\bxi, q, \bzeta]^T \in \vec{U} \times  H^1_{\#,*}(\Omega_b) \times (\vec{H}^1_{\#,\ast}(\Omega_f)\cap \vec{V})$, by the definition of $\mathcal D(\mathcal A^*)$ in Proposition \ref{thistheadjoint} we may indeed take $\widetilde{\bw} = \bxi,$ $\widetilde p = q$, and $\widetilde{\bv}=\bzeta$ in \eqref{tointegrate}.) In $\mathscr{D}'(0,T)$, we thus have, for given $[\widetilde{\bu},\widetilde{\bw},\widetilde p, \widetilde{\bv}]^T \in \mathcal D(\mathcal A^*)$,
\begin{align}
& \dfrac{d}{dt}\left\{ \rho _{b}(\mathbf{u}_{t},\widetilde{\mathbf{w}}%
)_{\Omega _{b}}+(c_{0}p,\widetilde{p})_{\Omega _{b}}+\rho _{f}(\mathbf{v},%
\widetilde{\mathbf{v}})_{\Omega _{f}}\right\} \nn  \nonumber \\
& = -(\sigma ^{E}(\vec{u}%
),\nabla \widetilde{\mathbf{w}})_{\Omega _{b}}-\alpha (\nabla \cdot \mathbf{u%
}_{t},\widetilde{p})_{\Omega _{b}}+\alpha (p,\nabla \cdot \widetilde{\mathbf{%
w}})_{\Omega _{b}}-(k\nabla p,\nabla \widetilde p)_{\Omega _{b}}-2\nu (\mathbf{D}(%
\mathbf{v}),\mathbf{D}(\widetilde{\mathbf{v}}))_{\Omega _{f}}\nn  \nonumber
\\
& -\beta ([\mathbf{v}-\mathbf{u}_{t}]\cdot \mathbf{\btau ,}[\widetilde{%
\mathbf{v}}-\widetilde{\mathbf{w}}]\cdot \mathbf{\btau )}_{\Gamma _{I}}+([%
\mathbf{v}-\mathbf{u}_{t}]\cdot \mathbf{e}_{3},\widetilde{p})_{\Gamma _{I}}-(%
{p,[\widetilde{\mathbf{v}}-\widetilde{\mathbf{w}}]\cdot \mathbf{e}}%
_{3})_{\Gamma _{I}}\nn  \nonumber \\
& +(\vec{F}_{b},\widetilde{\mathbf{w}})_{\Omega _{b}}+(S,\widetilde p)_{\Omega _{b}}+(\vec{F}%
_{f},\widetilde{\bv})_{\Omega _{f}}.  \label{8}
\end{align}

By way of simplifying the right hand side of \eqref{8}, we will now invoke the definition of $\mathcal D(\mathcal A^*)$. 
First, we have from the definition of $\mathcal A^*$ in Proposition \ref{thistheadjoint} the tangential slip condition and the stress-matching condition, invoked below:
\begin{align}\label{9}
 - \beta ([{\vec{v}-\bu_t]\cdot\btau,(\widetilde{\bv} - \widetilde{\bw})\cdot\btau})_{\Gamma_I} = & ~  ([{\vec{v}-\bu_t]\cdot\btau, [\sigma_f(\widetilde{\bv},\widetilde{\pi})\mathbf e_3]\cdot \btau})_{\Gamma_I}  \nn \\
 =& ~  ({\vec{v}\cdot\btau, [\sigma_f(\widetilde{\bv},\widetilde{\pi})\mathbf e_3]\cdot \btau})_{\Gamma_I }  \nn \\
 &~+  (\bu_t\cdot\btau, [\sigma_b(\widetilde{\bu},\widetilde{p})\mathbf e_3]\cdot \btau)_{\Gamma_I}. 
\end{align}

Secondly, we have from integrating by parts and the normal matching BC on $\Gamma_I$ in $\mathcal D(\mathcal A^*)$
\begin{align}\label{10}
 - (k\nabla p,\nabla \widetilde p)_{\Omega_b}- ({p,(\widetilde{\bv} - \widetilde{\bw})\cdot\vec{e}_3}))_{\Gamma_I}= & ~ - (k\nabla p,\nabla \widetilde p)_{\Omega_b}- ({p,k\nabla \widetilde p\cdot\vec{e}_3}))_{\Gamma_I}  \nn \\
 =& ~ c_0(p,kc_0^{-1}\Delta \widetilde p)_{\Omega_b}.
\end{align}

Lastly, again invoking the stress-matching BC on $\Gamma_I$ in $\mathcal D(\mathcal A^*)$:
\begin{align}\label{11}
\langle [{\vec{v}-\bu_t]\cdot\vec{e}_3,\widetilde p}\rangle_{\Gamma_I}= & ~\langle [{\vec{v}-\bu_t]\cdot\vec{e}_3,[\sigma_f(\widetilde{\bv},\widetilde{\pi})\mathbf e_3]\cdot\mathbf e_3 }\rangle_{\Gamma_I} \nn \\
= & ~\langle {\vec{v}\cdot\vec{e}_3,[\sigma_f(\widetilde{\bv},\widetilde{\pi})\mathbf e_3]\cdot\mathbf e_3 }\rangle_{\Gamma_I} + \langle {\vec{u}_t\cdot\vec{e}_3,[\sigma_b(\widetilde{\bu},\widetilde{p})\mathbf e_3]\cdot\mathbf e_3 }\rangle_{\Gamma_I} 
\end{align}
(Note that since $-[\sigma_b(\widetilde{\bu},\widetilde p)\mathbf e_3]\cdot \btau=\beta [\widetilde{\bv}-\widetilde{\bw}]\cdot \btau \in H^{1/2}(\Gamma_I)$,
then the second term on RHS of \eqref{11} is well-defined, via \eqref{trace}.

We rewrite the RHS of \eqref{8}, using the  the identities \eqref{9}--\eqref{11}:
                   \begin{align}\label{todus}
         \dfrac{d}{dt}\Big\{\rho_b(\vec{u}_{t},\widetilde{\bw})_{\Omega_b}+& \rho_f( \vec{v}(t),\widetilde{\bv})_{\Omega_f} + c_0(p(t), \widetilde p)_{L^2(\Omega_b)}\Big\}  \nn\\
        =&~\alpha(\vec{u}_t(t),  \nabla \widetilde p )_{\Omega_b} + \alpha \langle \bu_t(t)\cdot \mathbf e_3, \widetilde p \rangle_{\Gamma_I}  \nn \\ 
        & +c_0 (c_0^{-1}k p,\Delta \widetilde p)_{\Omega_b} - (\sigma^E(\vec{u}), \nabla \widetilde{\bw})_{\Omega_b}+\alpha(p,\nabla \cdot \widetilde{\bw})_{\Omega_b}  - 2\nu(\vec{D}(\vec{v}),\vec{D}(\widetilde{\bv}))_{\Omega_f}  \nn\\
        & + ({\vec{v}, [\sigma_f(\widetilde{\bv},\widetilde{\pi})\mathbf e_3]})_{\Gamma_I}+ (\bu_t, [\sigma_b(\widetilde{\bu},\widetilde{p})\mathbf e_3)_{\Gamma_I}  \nn \\
& +(\vec{F}_{b},\widetilde{\mathbf{w}})_{\Omega _{b}}+(S,\widetilde p)_{\Omega _{b}}+(\vec{F}%
_{f},\widetilde{\bv})_{\Omega _{f}}.  
    \end{align}

At this point, we note that for given $\left[ \widetilde{\mathbf{u}},%
\widetilde{\mathbf{w}},\widetilde{p},\widetilde{\mathbf{v}}\right] \in 
\mathcal{D}(\mathcal{A}^{\ast })$, one has the following point-wise
distributional relation with respect to Biot-Stokes solution component $%
\mathbf{u}_{t}$:%
\begin{equation}
-\left( \sigma ^{E}(\mathbf{u}_{t}),\nabla \widetilde{\mathbf{u}}\right)
_{\Omega _{b}}-\left\langle \mathbf{u}_{t},\sigma ^{E}(\widetilde{\mathbf{u}}%
)\cdot \mathbf{e}_{e}\right\rangle _{\Gamma _{I}}=\left( \mathbf{u}%
_{t},\nabla \cdot \sigma ^{E}\widetilde{\mathbf{u}}\right) _{\Omega _{b}}%
\text{ in }\mathscr{D}^{\prime }(0,T).  \label{green}
\end{equation}
With this relation in mind we add and subtract by $\left( \sigma ^{E}(\mathbf{u}_{t}),\nabla \widetilde{\mathbf{u}}\right)
_{\Omega _{b}}$, with respect to the right hand side of \eqref{todus}: This gives,

\begin{align}
\dfrac{d}{dt}\Big\{\rho _{b}(\vec{u}_{t},\widetilde{\bw})_{\Omega _{b}}+&
\rho _{f}(\vec{v}(t),\widetilde{\bv})_{\Omega _{f}}+c_{0}(p(t),\widetilde{p}%
)_{\Omega _{b}}\Big\}\nn  \notag \\
& =-(\sigma ^{E}(\vec{u}_{t}),\nabla \widetilde{\mathbf{u}})_{\Omega
_{b}}-\rho _{b}\left( \widetilde{\mathbf{u}}_{t},\rho _{b}^{-1}\nabla \cdot
\sigma ^{E}\widetilde{\mathbf{u}}\right) _{\Omega _{b}}- \cancel{\left\langle \mathbf{%
u}_{t},\sigma ^{E}(\widetilde{\mathbf{u}})\cdot \mathbf{e}_{3}\right\rangle
_{\Gamma _{I}} } -(\sigma ^{E}(\vec{u}),\nabla \widetilde{\bw})_{\Omega_{b}} 
\notag \\
& +\alpha \rho_{b}(\vec{u}_{t},\rho_{b}^{-1}\nabla \widetilde{p})_{\Omega_{b}}+ \cancel{ \alpha \langle \bu_{t}(t)\cdot \mathbf{e}_{3},\widetilde{p}\rangle
_{\Gamma _{I}} }+c_{0}(p,\alpha c_{0}^{-1}\nabla \cdot \widetilde{\bw}%
)_{\Omega_{b}}\nn  \notag \\
& +c_{0}(p,c_{0}^{-1}k\Delta \widetilde{p})_{\Omega_{b}}+\rho _{f}(\vec{v}%
,\nu \rho _{f}^{-1}\Delta (\widetilde{\bv}))_{\Omega _{f}}-\left( \mathbf{v},\nabla \widetilde{\pi }\right) _{\Omega _{f}} +\cancel{ \langle {\vec{u}%
_{t},\sigma_{b}(\widetilde{\bu},\widetilde{p})\mathbf{e}_{3}}\rangle
_{\Gamma _{I}} } \nn  \notag \\
& +(\vec{F}_{b},\widetilde{\mathbf{w}})_{\Omega _{b}}+(S,\widetilde{p}%
)_{\Omega_{b}}+(\vec{F}_{f},\widetilde{\bv})_{\Omega _{f}}.  \label{todus*}
\end{align}

In other words, if 
\begin{equation}
\mathbb{V}(t)\equiv \left[ 
\begin{array}{c}
\mathbf{u}(t) \\ 
\mathbf{u}_{t}(t) \\ 
p(t) \\ 
\mathbf{v}(t)%
\end{array}%
\right] \text{, \ }\mathbf{\Psi }\equiv \left[ 
\begin{array}{c}
\widetilde{\mathbf{u}} \\ 
\widetilde{\mathbf{w}} \\ 
\widetilde{p} \\ 
\widetilde{\mathbf{v}}%
\end{array}%
\right] \text{, and } ~~~\mathbf{F}\equiv \left[ 
\begin{array}{c}
0 \\ 
\vec{F}_{b} \\ 
S \\ 
\vec{F}_{f}%
\end{array}%
\right],  \label{dyn}
\end{equation}%
then from (\ref{todus*}) and the definition of $\mathcal{A}^{\ast }$ in Proposition \ref{thistheadjoint},
we have the point-wise a.e. $t$  relation,%
\begin{equation}
\frac{d}{dt}\left( \mathbb{V}(t),\mathbf{\Psi }\right) _{X}=\left( \mathbb{V}%
(t),\mathcal{A}^{\ast }\mathbf{\Psi }\right) _{X}+\left( \mathbf{F}(t),%
\mathbf{\Psi }\right) _{X}\text{, \ for every }\mathbf{\Psi }\in \mathcal D(\mathcal{A%
}^{\ast }).  \label{abs}
\end{equation}

Thus, for $\mathbb{V}(t)=\left[ \mathbf{u},\mathbf{u}_{t},p,\mathbf{v}\right]^T$, we have that
\begin{equation}
\begin{array}{l}
\mathbb{V}(t) \in C([0,T];\mathbf{U})\times C([0,T];\mathbf{L}^{2}(\Omega
_{b}))\times C([0,T];[H_{\#,\ast }^{1}(\Omega _{b})]^{\prime })\times
C([0,T]; \mathbf L^2(\Omega_f))%
\end{array}
\label{C_X}
\end{equation}%
satisfies the point-wise relation \eqref{abs} and is \textit{a fortiori} in $C([0,T];[\mathcal D(\mathcal{A}^{\ast })]^{\prime })$. If we now take
arbitrary $\bold\Psi $ in \eqref{abs} to be in $\mathcal D([\mathcal{A}^{\ast }]^{2})$, then by the standard semigroup theory \cite{pazy}, the  density of $\mathcal D([\mathcal A^*]^2)$ in $\mathcal D(\mathcal A^*)$ yields
that for a.e. $t$, the weak solution $\mathbb{V}(t)$ satisfies the Cauchy
problem,%
\begin{eqnarray*}
\frac{d}{dt}\mathbb{V}(t) &=&\mathcal{A}\mathbb{V}(t)+\mathbf{F}(t) \in 
[\mathcal D([\mathcal{A}^{\ast }]^{2})]^{\prime } \\
\mathbb{V}(0) &=&\left[ 
\begin{array}{c}
\mathbf{u}(0) \\ 
\mathbf{u}_{t}(0) \\ 
p(0) \\ 
\mathbf{v}(0)%
\end{array}%
\right] \in X\subset \lbrack D(\mathcal{A}^{\ast })]^{\prime }.
\end{eqnarray*}%
As such, $\mathbb{V}(t)$ is necessarily given explicitly via the agency of already the established
Biot-Stokes semigroup $\left\{ e^{\mathcal{A}t}\right\} _{t\geq 0}$ in \cite[Theorem 1]{oldpaper}. That is, invoking the variation of parameters representation of the solution \cite{pazy},
\begin{equation}
\mathbb{V}(t)=e^{\mathcal{A}t}\mathbb{V}(0)+\int_{0}^{t}e^{\mathcal{A}(t-s)}%
\mathbf{F}(t)ds,  \label{char}
\end{equation}
\noindent with the given right hand side being in $C([0,T];X)$, not just $C([0,T];[D(%
\mathcal{A}^{\ast })]^{\prime })$. And so, in particular, we establish that the Biot-Stokes
weak solution component $p$ is indeed in $C([0,T];L^{2}(\Omega _{b}))$. The
characterization of the weak solution (\ref{char}) through the variation of parameters formula then finishes the proof of the case $c_{0}>0$. Namely, weak solutions in the sense of Definition \ref{weaksols} must coincide with semigroup solutions as established in \cite[Theorem 1 and Corollary 2.2]{oldpaper}; uniqueness of semigroup solutions may then be invoked, and so this concludes the proof of the main result here, Theorem \ref{th:main}. 
\end{proof}

\section*{Appendix: Adjoint Calculation}\label{adjoint}
In this section we prove the characterization of the adjoint of the semigroup generator $\mathcal A$ presented in Section \ref{operator}. Specifically, this section serves to prove Proposition \ref{thistheadjoint}. 

We begin by defining a set $\mathcal S$ which carries regularity information and interfacial boundary conditions, and will be used to ascertain $\mathcal D(A^*)$. This set has the same regularity as $\mathcal D(A)$ for the first nine bullets of Definition \ref{diffdomain}, but its interface conditions on $\Gamma_I$ vary from that of $D(A)$ by several sign changes. 
\begin{definition} \label{sdomain} Let $\vec{y}\in X$. Then $\vec{y} = [\vec{u},\vec{w},p,\vec{v}]^T \in \mathcal S$ if and only if the following bullets hold:
\begin{itemize}
    \item $\vec{u} \in \vec{U}$  with  $\mathcal{E}_0(\vec{u}) \in \vec{L}^2(\Omega_b)$ (so that $[{\sigma_b(\vec{u})\mathbf n}]\big|_{\partial\Omega_b} \in \vec{H}^{-1/2}(\partial \Omega_b)$);
    \item $\vec{w} \in \vec{U}$;
    
    \item $p \in  H_\#^1(\Omega_b)$ with $A_0p \in L^2(\Omega_b)$ (so that $\left.\partial_{\mathbf n}p\right|_{\partial\Omega_b} \in {H}^{-1/2}(\partial \Omega_b)$);
    
        \item $[{\sigma_b(\vec{u})\mathbf n}] \in \vec{H}_{\#}^{-1/2}(\partial \Omega_b)$ (then  
       $\sigma^E(\bu)\mathbf n\cdot \mathbf n \in H^{-1/2}_{\#}(\partial \Omega_b)$);
       \item  $\partial_{\mathbf n} p \in H^{-1/2}_{\#}(\partial \Omega_b)$
    
    \item $\vec{v} \in \vec{H}_\#^1(\Omega_f) \cap \vec{V}$ with $\vec{v}|_{\Gamma_f} = \vec{0}$;
    \item There exists $\pi \in L^2(\Omega_f)$ such that $$\nu\Delta \vec{v} - \nabla\pi \in \vec{V},$$ where $\pi=\pi(p,\bv)$ is as in \eqref{pi*} (and so
 $\left.\sigma_f\mathbf(\bv,\pi) n\right|_{\partial \Omega_f}\in \vec{H}^{-\frac{1}{2}}(\partial\Omega_f)$ and  $\left.\pi\right|_{\partial \Omega_f} \in H^{-\frac{1}{2}}(\partial\Omega_f)$;
   
    \item $2\nu \mathbf D(\bv) \big|_{\partial \Omega_f} \in \mathbf H^{-1/2}_{\#}(\partial \Omega_f)$ and ~$\pi \big|_{\partial\Omega_f} \in H^{-1/2}_{\#}(\partial \Omega_f)$;
    
    \item  $\left.\Delta \vec{v}\cdot\mathbf n\right|_{\partial \Omega_f} \in H^{-3/2}(\partial\Omega_f)$);
    
    \item  $\restri{(\vec{v} - \vec{w})\cdot\vec{e}_3} = \restri{k\nabla p\cdot\vec{e}_3} \in H^{-1/2}(\Gamma_I)$;
    
    \item $\restri{\beta(\vec{v}-\vec{w})\cdot\btau} = \restri{-\btau\cdot\sigma_f\vec{e}_3} \in H^{-1/2}(\Gamma_I)$;
    
    \item $\restri{\vec{e}_3 \cdot \sigma_f\vec{e}_3} = \restri{p} \in H^{-1/2}(\Gamma_I)$;

    \item $-\restri{\sigma_b\vec{e}_3} = \restri{\sigma_f\vec{e}_3} \in \vec{H}^{-1/2}(\Gamma_I)$.
\end{itemize}
\end{definition}
Let us also define an operator $\mathcal L$ on $\mathcal D(\mathcal A)$, which will later be identified with $\mathcal A^*$. This operator is characterized by the adjoint action of $\mathcal A$:
\begin{equation}\label{lop}
(\mathcal A \by, \byt)_{X} = (\by,\mathcal L\byt)_X,~~\forall ~\by, ~\byt \in \mathcal D(\mathcal A).
\end{equation}

Now, we prove two Lemmas:
\begin{lemma}
The set $\mathcal S$ as in Definition \ref{sdomain} is a subset of $\mathcal D(\mathcal A^*)$.
\end{lemma}
\begin{proof}
Let $\by = [\bu, \bw, p, \bv]^T \in \mathcal D(\mathcal A)$ and $\byt = [\widetilde{\bu}, \widetilde{\bw}, \widetilde p, \widetilde{\bv}]^T \in \mathcal S$. Then we compute $(\mathcal A\by,\byt)_X$:
\begin{align}
(\mathcal A\by,\byt)_X = & ~ (\sigma^E(\bw), \mathbf D(\widetilde{\bu}))_{\Omega_b} +\rho_b\rho_b^{-1}(\text{div}~\sigma^E(\bu),\widetilde{\bw})_{\Omega_b}-\alpha \rho_b\rho_b^{-1}(\nabla p, \widetilde{\bw})_{\Omega_b}\\ \nonumber
& -\alpha c_0c_0^{-1}(\nabla \cdot \bw, \widetilde p)_{\Omega_b}+c_0c_0^{-1}k(\Delta p, \widetilde p)_{\Omega_b} \\ \nonumber
& + \rho_f\rho_f^{-1}\nu(\Delta \bv, \widetilde{\bv})_{\Omega_f} - \rho_f\rho_f^{-1}(\nabla \pi, \widetilde{\bv})_{\Omega_f},
\end{align}
where $\pi$ is as in \eqref{pi} .  We proceed with the calculation, invoking integration by parts (generalized Green's theorem):
\begin{align}
(\mathcal A\by,\byt)_X = & ~ (\bw,-\text{div}~\sigma^E(\widetilde{\bu}))_{\Omega_b}+\langle \bw, \sigma^E(\widetilde{\bu})  \mathbf  n \rangle_{\Gamma_I}- (\sigma^E(\bu), \mathbf D(\widetilde{\bw}))_{\Omega_b}+\langle \sigma^E(\bu)\mathbf n, \widetilde{\bw}\rangle_{\Gamma_I} \\ \nonumber
&-\alpha\langle p\mathbf n , \widetilde{\bw} \rangle_{\Gamma_I}+\alpha (p , \nabla \cdot \widetilde{\bw})_{\Omega_b} - \alpha \langle \bw, \widetilde p \mathbf n\rangle_{\Gamma_I} + \alpha (\bw, \nabla \widetilde p)_{\Omega_b} + k \langle \nabla p\cdot \mathbf n, \widetilde p\rangle_{\Gamma_I}\\ \nonumber
& -k(\nabla p, \nabla \widetilde p)_{\Omega_b}-\nu(\nabla \bv, \nabla \widetilde{\bv})_{\Omega_f}-\nu\langle [\nabla \bv] \mathbf n ,  \widetilde{\bv}\rangle_{\Gamma_I}+\langle \pi \mathbf n , \widetilde{\bv}\rangle_{\Gamma_I} .
\end{align}
Above, the normal vector $\mathbf n=-\mathbf e_3$ is taken with respect to  $\Omega_f$; of course, we have also invoked lateral periodicity on $\Gamma_{\text{lat}}$, as well as Dirichlet conditions $\bu \equiv 0$ and $p \equiv 0$ on $\Gamma_b$ and $\bv = 0$ on $\Gamma_f$. Applying the divergence theorem and Green's theorem once more, we arrive at the central identity for this lemma:
\begin{align} \label{ayy}
(\mathcal A\by,\byt)_X = & -(\sigma^E(\bu),\mathbf D(\widetilde{\bw}))_{\Omega_b} +\rho_b(\bw,-\rho_b^{-1}\text{div} ~\sigma^E(\widetilde{\bu}))_{\Omega_b} +c_0(p,\alpha c_0^{-1}\nabla\cdot\widetilde{\bw})_{\Omega_b}
\\ 
&+\alpha\rho_b(\bw,\rho_b^{-1}\nabla \widetilde p)_{\Omega_b} +c_0(p,c_0^{-1}k\Delta \widetilde p)_{\Omega_b} +\rho_f(\bv, \rho_f^{-1}\nu\Delta \widetilde{\bv})_{\Omega_f} + B.T.
\nonumber\\
\text{with}~~B.T.~ \equiv & -\alpha \langle \bw  , \widetilde p \mathbf n\rangle_{\Gamma_I} - \langle  p, k\nabla \widetilde p \cdot \mathbf n \rangle_{\Gamma_I}+\nu \langle \bv , [\nabla \widetilde{\bv}] \mathbf n \rangle_{\Gamma_I}+\langle \bw , \sigma^E(\widetilde{\bu})\mathbf n \rangle_{\Gamma_I} \\ 
\nonumber
& +\langle \sigma^E(\bu)\mathbf n, \widetilde{\bw}\rangle_{\Gamma_I}-\alpha \langle p\mathbf n , \widetilde{\bw} \rangle_{\Gamma_I} + k\langle \nabla p \cdot \mathbf n , \widetilde p \rangle_{\Gamma_I}- \nu\langle [\nabla  \bv]\mathbf n , \widetilde{\bv} \rangle_{\Gamma_I}+\langle \pi \mathbf n, \widetilde{\bv}\rangle_{\Gamma_I}.
\end{align}
We proceed to address the pressure, $\widetilde{\pi}$, as associated to $\byt \in \mathcal S$. This is to say, there is a $\widetilde{\pi} \in L^2(\Omega_f)$ so that
$$\nu\Delta \widetilde{\bv}-\nabla \widetilde{\pi} \in \bf V.$$ Thence, we may rewrite $B.T.$ in terms of $\widetilde \pi$: 
\begin{align}
B.T. = & ~ \langle \bw, \sigma_b(\widetilde{\bu},\widetilde p)\mathbf n \rangle_{\Gamma_I} - \langle p , k\nabla \widetilde p \cdot \mathbf n \rangle_{\Gamma_I} +\langle \sigma_b(\bu,p)\mathbf n, \widetilde{\bw}\rangle_{\Gamma_I}+k\langle \nabla p\cdot\mathbf n, \widetilde p\rangle_{\Gamma_I} \\ \nonumber
& -\langle \sigma_f(\bv,\pi)\mathbf n, \widetilde{\bv}\rangle_{\Gamma_I}+\nu\langle \bv, [\nabla \widetilde{\bv}]\mathbf n\rangle_{\Gamma_I} - \langle \bv, \widetilde \pi \mathbf n\rangle_{\Gamma_I}+ \langle \bv, \widetilde \pi \mathbf n\rangle_{\Gamma_I}, \\ \nonumber
 \text{ (where we have} &  \text{ added and subtracted the last terms involving $\widetilde{\pi}$) } \\ \label{rewriz}
= & ~ -\langle \bw, \sigma_b(\widetilde{\bu},\widetilde p)\mathbf e_3 \rangle_{\Gamma_I} +\langle p , k\nabla \widetilde p \cdot \mathbf e_3 \rangle_{\Gamma_I} -\langle \sigma_b(\bu,p)\mathbf e_3, \widetilde{\bw}\rangle_{\Gamma_I} \\ \nonumber
&-k\langle \nabla p\cdot\mathbf e_3, \widetilde p\rangle_{\Gamma_I}+\langle \sigma_f(\bv,\pi)\mathbf e_3, \widetilde{\bv}\rangle_{\Gamma_I} -\langle \bv, \sigma_f(\widetilde{\bv},\widetilde \pi)\mathbf e_3\rangle_{\Gamma_I}-(\bv, \nabla \widetilde \pi)_{\Omega_f}.
\end{align}
Now, with the definition of $B.T.$ above, and recalling the interface conditions on $\Gamma_I$ associated to $\mathcal D(\mathcal A)$ in Definition 2, we can rewrite \eqref{rewriz} as:
\begin{align} \label{bount}
B.T. +(\bv, \nabla \widetilde \pi)_{\Omega_f}= &~ -\langle \bw\cdot \mathbf e_3, [\sigma_b(\widetilde{\bu},\widetilde p)\mathbf e_3]\cdot \mathbf e_3\rangle_{\Gamma_I} -\langle \bw\cdot \btau, [\sigma_b(\widetilde{\bu},\widetilde p)\mathbf e_3]\cdot \mathbf \btau\rangle_{\Gamma_I} \\ \nn
& -\langle [\sigma_f(\bv,\pi)\mathbf e_3]\cdot \mathbf e_3, k\nabla \widetilde p\cdot \mathbf e_3 \rangle_{\Gamma_I}  -\langle [\sigma_f(\mathbf v, \pi)\mathbf e_3]\cdot \mathbf e_3, \widetilde{\bw}\cdot \mathbf e_3 \rangle_{\Gamma_I} \\ \nn
& - \langle [\sigma_f(\mathbf v, \pi)\mathbf e_3]\cdot \mathbf \btau, \widetilde{\bw}\cdot \btau \rangle_{\Gamma_I} +\langle [\bv-\bw]\cdot \mathbf e_3, \widetilde p \rangle_{\Gamma_I} +\langle[\sigma_f(\bv,\pi)\mathbf e_3]\cdot \mathbf e_3 , \widetilde{\bv}\cdot \mathbf e_3 \rangle_{\Gamma_I}\\ \nn
&+\langle [\sigma_f(\bv,\pi)\mathbf e_3]\cdot \btau , \widetilde{\bv}\cdot \btau\rangle_{\Gamma_I} - \langle \bv\cdot \mathbf e_3, [\sigma_f(\widetilde{\bv},\widetilde \pi)\mathbf e_3]\cdot \mathbf e_3\rangle_{\Gamma_I} - \langle \bv \cdot \btau, [\sigma_f(\widetilde{\bv},\widetilde \pi)\mathbf e_3]\cdot\btau\rangle_{\Gamma_I}.
\end{align}
We now invoke the boundary conditions on $\Gamma_I$ in in Definition 2 (for $\mathcal{D}(\mathcal{A})$
) and Definition \ref{sdomain} (for $\mathcal S$,  with respect to $\by$ and $\byt$. We begin with the tangential terms in \eqref{bount}:
\begin{align}
 &- \langle \bw\cdot \btau, [\sigma_b(\widetilde{\bu},\widetilde p)\mathbf e_3 ]\cdot \mathbf \btau\rangle_{\Gamma_I} 
   - \langle [\sigma_f(\mathbf v, \pi)\mathbf e_3]\cdot \mathbf \btau, \widetilde{\bw}\cdot \btau \rangle_{\Gamma_I}   \\\nonumber
& +\langle [\sigma_f(\bv,\pi)\mathbf e_3]\cdot \btau , \widetilde{\bv}\cdot \btau\rangle_{\Gamma_I}   - \langle \bv \cdot \btau, [\sigma_f(\widetilde{\bv},\widetilde \pi)\mathbf e_3]\cdot\btau\rangle_{\Gamma_I}\\ \nn
=&-\left\langle [\mathbf{v}-\mathbf{w}]\cdot \mathbf{\tau },[\sigma _{f}(%
\widetilde{\mathbf{v}},\widetilde{\mathbf{w}})\mathbf{e}_{3}]\cdot \mathbf{%
\tau }\right\rangle _{\Gamma _{I}}+\left\langle [\sigma _{f}(\mathbf{v},%
\mathbf{w})\mathbf{e}_{3}]\cdot \mathbf{\tau },[\widetilde{\mathbf{v}}-%
\widetilde{\mathbf{w}}]\cdot \mathbf{\tau }\right\rangle _{\Gamma _{I}} \\ \nn
= & ~ -\beta \langle [\bv-\bw]\cdot \btau, [\widetilde{\bv}-\widetilde{\bw}]\cdot \btau \rangle_{\Gamma_I} + \beta \langle  [\bv-\bw]\cdot \btau, [\widetilde{\bv} -\widetilde{\bw}]\cdot\btau \rangle_{\Gamma_I} \\\nonumber
=& ~0.
\end{align}
Proceeding similarly, we have:
\begin{align}\nonumber
 &- \langle \bw\cdot \mathbf e_3, [\sigma_b(\widetilde{\bu},\widetilde p)\mathbf e_3]\cdot \mathbf e_3 \rangle_{\Gamma_I} - \langle \bv \cdot \mathbf e_3 , [\sigma_f(\widetilde{\bv},\widetilde \pi)\mathbf e_3]\cdot \mathbf e_3 \rangle_{\Gamma_I} + \langle [\bv -\bw]\cdot \mathbf e_3, \widetilde p \rangle_{\Gamma_I} \\
=&~ \langle \mathbf w \cdot \mathbf e_3 , [\sigma_f(\widetilde{\bv},\widetilde \pi)\mathbf e_3]\cdot \mathbf e_3 \rangle_{\Gamma_I} - \langle \bv\cdot \mathbf e_3, [\sigma_f(\widetilde{\bv},\widetilde \pi)\mathbf e_3]\cdot \mathbf e_3 \rangle_{\Gamma_I} +\langle [\bv -\bw]\cdot \mathbf e_3 , [\sigma_f(\widetilde{\bv},\widetilde \pi)\mathbf e_3]\cdot \mathbf e_3 \rangle_{\Gamma_I} \\\nonumber
=& ~0; \nonumber \\ 
\text{and} & \nonumber \\ \nonumber
& - \langle [\sigma_f(\bv, \pi)\mathbf e_3]\cdot \mathbf e_3 , k\nabla \widetilde p\cdot \mathbf e_3\rangle_{\Gamma_I} - \langle [\sigma_f(\bv,\pi)\mathbf e_3]\cdot \mathbf e_3, \widetilde{\bw}\cdot \mathbf e_3\rangle_{\Gamma_I}+ \langle [\sigma_f(\bv, \pi)\mathbf e_3]\cdot \mathbf e_3 , \widetilde{\bv}\cdot \mathbf e_3\rangle_{\Gamma_I} \\ \nonumber
=& ~ - \langle [\sigma_f(\bv, \pi)\mathbf e_3]\cdot \mathbf e_3 , [\widetilde{\bv}-\widetilde{\bw}]\cdot \mathbf e_3\rangle_{\Gamma_I} - \langle [\sigma_f(\bv,\pi)\mathbf e_3]\cdot \mathbf e_3, \widetilde{\bw}\cdot \mathbf e_3\rangle_{\Gamma_I}+ \langle [\sigma_f(\bv, \pi)\mathbf e_3]\cdot \mathbf e_3 , \widetilde{\bv}\cdot \mathbf e_3\rangle_{\Gamma_I} \nonumber \\
=&~ 0.
\end{align}
Applying the calculations (3.54)-(3.56) to the right hand side of \eqref{bount}, we then shown
\begin{equation}\label{Bt} B.T. +(\bv, \nabla \widetilde{\pi})_{\Omega_f} = 0.\end{equation}
In turn, applying \eqref{Bt} to \eqref{ayy}, we conclude: ~$\forall~\byt \in S$ and $y \in \mathcal D(\mathcal A)$,
$$(\mathcal A \by, \byt)_X = (\by, \mathcal L\byt)_X,$$
with $\mathcal L$ having action
\begin{equation}\label{ell}
\mathcal L \byt = \begin{bmatrix} 
- \widetilde{\bw} \\ \mathcal E_0\widetilde{\bu}+ \alpha \rho_b^{-1}\nabla \widetilde p \\ \alpha c_0^{-1}\nabla \cdot \widetilde{\bw}-A_0p \\ \rho_f^{-1}\nu\Delta \widetilde{\bv}-\rho_f^{-1}\nabla \widetilde \pi
\end{bmatrix}
\end{equation} (after recalling \eqref{ops}). In this case, the pressue $\widetilde \pi$ is identified by the boundary value problem
\begin{equation} \label{pis3} \begin{cases}
\Delta \widetilde \pi = 0 & \text{ in } ~\Omega_f \\
\partial_{\mathbf n} \widetilde \pi = \nu \Delta \widetilde{\bv}\cdot \mathbf n & \text{ on }~\Gamma_f \\
\widetilde \pi = [\sigma_b(\widetilde \bu, \widetilde p)\mathbf e_3]\cdot \mathbf e_3+[\nu\nabla \bv\cdot \mathbf e_3]\cdot \mathbf e_3 & \text{ on }~\Gamma_I \\
\widetilde \pi~~\text{ is (Dirichlet) periodic on }~\Gamma_{\text{lat}}.
\end{cases}\end{equation}
\end{proof}
We have now shown that $\mathcal S \subseteq \mathcal D(\mathcal A^*)$ and that $$\mathcal A^*\big|_{\mathcal S} = \mathcal L.$$ We will finish the characterization of $\mathcal A^*$ in Proposition \ref{thistheadjoint} if we can show that $\mathcal D(\mathcal A^*) \subseteq \mathcal S.$ We now proceed with an intermediate result.
\begin{lemma} \label{resolvent}
For operator $\mathcal L$ given by \eqref{ell} with $\mathcal D(\mathcal L) =\mathcal S$, we have that 
$$0 \in \rho(\mathcal L);$$
i.e., $\mathcal L$ is boundedly invertible on $X$. 
\end{lemma}
\begin{proof}
We will show that $\mathcal L : \mathcal S \subset X \to X$ given by \eqref{ell} is boundedly invertible. Let $\by^* = [\mathbf w_1^*, \mathbf w_2^*, q^*, \mathbf f^*]^T \in X$ and consider the equation
$$\mathcal L \by = \by^*.$$ 
We must find $\by = [\bu, \bw, p, \bv]^T \in \mathcal D(\mathcal L)$ so that

\begin{equation} \label{lmess} \begin{cases}
-\bw = \bw_1^* \\
-\rho_b^{-1}\text{div}~\sigma^E(\bu) + \alpha \rho_b^{-1}\nabla p = \bw^*_2 \\
-\alpha c_0^{-1}\nabla \cdot \bw_1^*+c_0^{-1}k\Delta p = q^* \\ 
\rho_f^{-1}\nu\Delta \bv-\rho_f^{-1}\nabla \pi = \mathbf f^*,
\end{cases}\end{equation}
including the other defining characteristics for $\by \in \mathcal D(\mathcal L)$. As before, $\pi$ here, given explicitly by \eqref{pi*}, will be the associated Stokes pressure for the variable $\by$. 

Now, consider test functions $[\widetilde p, \widetilde{\bv}]^T \in H^1_{\#,*}(\Omega_b) \times \mathbf H^1_{\#,*}(\Omega_f)$ for the last two equations in \eqref{lmess}. Testing brings about
\begin{align*}
-k(\nabla p, \nabla \widetilde p)_{\Omega_b} - \langle [\bv+\bw_1^*]\cdot \mathbf e_3, \widetilde p\rangle_{\Gamma_I}=  \alpha (\nabla \cdot \bw_1^*, \widetilde p)_{\Omega_b}& ~ +c_0(q^*, \widetilde p)_{\Omega_b}; \\
-\nu(\nabla \bv, \nabla \widetilde{\bv})_{\Omega_f} + \langle [\sigma_f(\bv,\pi)\mathbf e_3]\cdot \mathbf e_3, \widetilde{\bv}\cdot \mathbf e_3\rangle_{\Gamma_I}+\langle [\sigma_f(\bv,\pi)\mathbf e_3]\cdot \btau, \widetilde{\bv}\cdot \btau\rangle_{\Gamma_I}+(\pi,\nabla \cdot \widetilde{\bv})_{\Omega_f}= & ~ \rho_f(\mathbf f^*, \widetilde{\bv})_{\Omega_f}.
\end{align*}

Applying the boundary conditions on $\mathcal D(\mathcal L)$ and rewriting, we have
\begin{equation} \label{mixedvar}  \begin{cases}
-k(\nabla p, \nabla \widetilde p)_{\Omega_b} - \langle \bv\cdot \mathbf e_3, \widetilde p\rangle_{\Gamma_I}=\langle \bw_1^*\cdot \mathbf e_3,\widetilde p\rangle_{\Gamma_I}+  \alpha (\nabla \cdot \bw_1^*, \widetilde p)_{\Omega_b}  +c_0(q^*, \widetilde p)_{\Omega_b}; \\
-\nu(\nabla \bv, \nabla \widetilde{\bv})_{\Omega_f} + \langle p , \widetilde{\bv}\cdot \mathbf e_3\rangle_{\Gamma_I}-\beta \langle \mathbf v\cdot \btau , \widetilde{\bv}\cdot \btau\rangle_{\Gamma_I}+(\pi,\nabla \cdot \widetilde{\bv})_{\Omega_f}=  ~\beta \langle \bw_1^*\cdot \btau, \widetilde{\bv}\cdot \btau\rangle_{\Gamma_I}+  \rho_f(\mathbf f^*, \widetilde{\bv})_{\Omega_f}.
\end{cases} \end{equation}
The above equations \eqref{mixedvar} constitute the following \textbf{mixed-variational problem}:
\begin{quote}
Find $\big\{[p,\bv], \pi\big\} \in [H^1_{\#,*}(\Omega_b) \times \mathbf H^1_{\#,*}(\Omega_f)] \times L^2(\Omega_f)$ such that,
$$k(\nabla p, \nabla \widetilde p)_{\Omega_b} + \langle \bv\cdot \mathbf e_3, \widetilde p\rangle_{\Gamma_I}=-\langle \bw_1^*\cdot \mathbf e_3,\widetilde p\rangle_{\Gamma_I}-  \alpha (\nabla \cdot \bw_1^*, \widetilde p)_{\Omega_b} -c_0(q^*, \widetilde p)_{\Omega_b}$$
holds for all $\widetilde p \in H^1_{\#,*}(\Omega_b)$;
$$\nu(\nabla \bv, \nabla \widetilde{\bv})_{\Omega_f} - \langle p , \widetilde{\bv}\cdot \mathbf e_3\rangle_{\Gamma_I}+\beta \langle \mathbf v\cdot \btau , \widetilde{\bv}\cdot \btau\rangle_{\Gamma_I}-(\pi,\nabla \cdot \widetilde{\bv})_{\Omega_f}=  -\beta \langle \bw_1^*\cdot \btau, \widetilde{\bv}\cdot \btau\rangle_{\Gamma_I}-  \rho_f(\mathbf f^*, \widetilde{\bv})_{\Omega_f}$$
holds for all $\widetilde{\bv} \in \mathbf H^1_{\#,*}(\Omega_f)$; 
$$(\widetilde q, \nabla \cdot \bv)_{\Omega_f} = 0,$$
holds for all $\widetilde q \in L^2(\Omega_f)$. \end{quote}
This system is entirely analogous to the maximality part of the proof in \cite{oldpaper}, namely \cite[Sections 4.2.1--4.2.2]{oldpaper}. Thus, by the Babu\v{s}ka-Brezzi Theorem, we obtain a solution $\big\{[p,\bv], \pi\big\} \in [H^1_{\#,*}(\Omega_b) \times \mathbf H^1_{\#,*}(\Omega_f)] \times L^2(\Omega_f)$. With that solution in hand, we may specify test function $\widetilde p \in \mathcal{D}(\Omega _{b})$ above to obtain (via integration by parts and identifying $-\bw = \bw_1^*$) that
\begin{equation} \label{interior1} k\Delta p +\alpha \nabla \cdot \bw = c_0 q^*~\text{ in }~\Omega_b.\end{equation}
With $\Delta p =c_0q^*-\alpha \nabla \cdot \bw_1^* \in L^2(\Omega_b)$ and $p \in H^1_{\#,*}(\Omega_b)$, we further infer from the generalized Green's theorem
\begin{equation}\label{estt1}
||k\nabla p \cdot \mathbf n||_{H^{-1/2}(\partial\Omega_b)} \le C ||\by^*||_X.
\end{equation}
We proceed similarly, taking $\widetilde{\bv} \in [\mathcal{D}(\Omega _{f})]^3$ in the said mixed variational problem associated to \eqref{mixedvar}, obtaining
\begin{equation}\label{interior2} - \nu\Delta \bv - \nabla \pi = \rho_f\mathbf f^* \in \mathbf{L}^2(\Omega_f).\end{equation}
With $\mathbf f^* \in \mathbf V$ and $\bv \in \mathbf H^1_{\#,*}(\Omega_f)$, we obtain the estimate (proceeding as in \cite[Section 4.2.2]{oldpaper} for the boundary traces)
\begin{align}\label{estt2}
||\nu[\nabla \bv] \mathbf n||_{\mathbf H^{-1/2}(\partial\Omega_f)}+||\pi||_{H^{-1/2}(\partial \Omega_f)}+||\partial_{\mathbf n}\pi||_{H^{-3/2}(\partial \Omega_f)}+||\nabla \bv||_{[L^2(\Omega_f)]^{3\times 3}}+||\pi||_{L^2(\Omega_f)} \le C||\by^*||_X
\end{align}

\smallskip

Since we have $\mathbf f^* \in \mathbf V$, we know $\mathbf f^*\cdot \mathbf n\big|_{\partial \Omega_f} \in \mathbf H^{-1/2}(\partial \Omega_f)$ (with associated estimate). Consequently, we can read off from \eqref{interior2} (using \eqref{estt2}),
\begin{equation}\label{estt3}
||\nu \Delta \bv \cdot \mathbf e_3||_{H^{-3/2}(\Gamma_I)} \le ||\nabla \pi \cdot \mathbf n||_{H^{-3/2}(\partial \Omega_f)}+\rho_f ||\mathbf f^*\cdot \mathbf n||_{H^{-1/2}(\partial \Omega_f)} \le C||\by^*||_{X}.
\end{equation}
Moreover, directly repeating the argument in \cite[Section 4.2.2]{oldpaper}, we obtain the lateral trace periodicities,
\begin{align}\label{pisbound}
& \pi\big|_{\partial \Omega_f} \in H^{-1/2}_{\#}(\partial \Omega_f),~~[\nabla \bv]\mathbf n \big|_{\partial \Omega_f} \in \mathbf H^{-1/2}_{\#}(\partial \Omega_f),\\ \nn  &~~\nabla p \cdot \mathbf n \in H^{-1/2}_{\#}(\partial \Omega_b). 
\end{align}
With these justified (generalized) boundary traces in hand, we return to the mixed-variational problem associated to \eqref{mixedvar}. We invoke Green's theorem again in second order and divergence terms to obtain
\begin{equation} \label{mixedvar*}  \begin{cases}
k(\Delta p,  \widetilde p)_{\Omega_b}-k\langle \nabla p\cdot \mathbf n, \widetilde p\rangle_{\Gamma_I} - \langle \bv\cdot \mathbf e_3, \widetilde p\rangle_{\Gamma_I}&=  \langle \bw_1^*\cdot \mathbf e_3,\widetilde p\rangle_{\Gamma_I}-  \alpha (\nabla \cdot \bw, \widetilde p)_{\Omega_b}  +c_0(q^*, \widetilde p)_{\Omega_b} \\
\nu(\Delta \bv,  \widetilde{\bv})_{\Omega_f} + \nu\langle [\nabla \bv]\mathbf n,\widetilde{\bv}\rangle_{\Gamma_I} + \langle p , \widetilde{\bv}\cdot \mathbf e_3\rangle_{\Gamma_I} & -\beta \langle \mathbf v\cdot \btau , \widetilde{\bv}\cdot \btau\rangle_{\Gamma_I} -(\nabla \pi,\widetilde{\bv})_{\Omega_f} -\langle \pi\mathbf n, \widetilde{\bv}\rangle_{\Gamma_I} \\
& =    \beta \langle \bw_1^*\cdot \btau, \widetilde{\bv}\cdot \btau\rangle_{\Gamma_I}+  \rho_f(\mathbf f^*, \widetilde{\bv})_{\Omega_f},
\end{cases} \end{equation}
for all $\widetilde p \in H^1_{\#,*}(\Omega_b),~\widetilde{\bv} \in \mathbf H^1_{\#,*}(\Omega_f)$. 
With the pointwise relations \eqref{interior1} and \eqref{interior2}, we cancel interior terms to obtain
\begin{equation} \label{mixedvarboundary}  \begin{cases}
-k\langle \nabla p\cdot \mathbf n, \widetilde p\rangle_{\Gamma_I} - \langle \bv\cdot \mathbf e_3, \widetilde p\rangle_{\Gamma_I}=  \langle \bw_1^*\cdot \mathbf e_3,\widetilde p\rangle_{\Gamma_I}   \\
 \nu\langle [\nabla \bv]\mathbf n,\widetilde{\bv}\rangle_{\Gamma_I} + \langle p , \widetilde{\bv}\cdot \mathbf e_3\rangle_{\Gamma_I}  -\beta \langle \mathbf v\cdot \btau , \widetilde{\bv}\cdot \btau\rangle_{\Gamma_I}  -\langle \pi\mathbf n, \widetilde{\bv}\rangle_{\Gamma_I}
 =    \beta \langle \bw_1^*\cdot \btau, \widetilde{\bv}\cdot \btau\rangle_{\Gamma_I},
\end{cases} \end{equation}
for all $\widetilde p \in H^1_{\#,*}(\Omega_b),~\widetilde{\bv} \in \mathbf H^1_{\#,*}(\Omega_f)$.  Rewriting, we obtain
\begin{align} \label{var1bd}
k\langle \nabla p\cdot \mathbf e_3 ,\widetilde p\rangle_{\Gamma_I} =& ~ \langle [\bv - \bw]\cdot \mathbf e_3, \widetilde p\rangle_{\Gamma_I}, \\  \label{var2bd}
-\langle \sigma_f(\bv,\pi)\mathbf e_3, \widetilde{\bv}\rangle_{\Gamma_I} =&~ -\langle p, \widetilde{\bv}\cdot \mathbf e_3\rangle_{\Gamma_I} + \beta\langle[\bv-\bw]\cdot \btau, \widetilde{\bv}\cdot\btau\rangle_{\Gamma_I}, 
\end{align}
for all $\widetilde p \in H^1_{\#,*}(\Omega_b),~\widetilde{\bv} \in \mathbf H^1_{\#,*}(\Omega_f)$. 
To recover the boundary conditions from the variational relations, we proceed as in \cite{oldpaper}. Namely, let $g \in H_0^{1/2+\epsilon}(\Gamma_I)$ and extend by zero ot the rest of $\partial\Omega_b$, invoking the surjectivity of the trace theorem to obtain $\widetilde p \in H^1(\Omega_b)$, so that 
$$\widetilde p \big|_{\partial \Omega_b} = \begin{cases} g & \text{ on }~\Gamma_I \\ 0 & \text{ on }~\partial\Omega_b\setminus \overline{\Gamma_I} .\end{cases}$$
Testing \eqref{var1bd} with this $\widetilde p$ yields
$$k\langle \nabla p\cdot \mathbf e_3 , g\rangle_{\Gamma_I} = \langle [\bv-\bw]\cdot \mathbf e_3, \widetilde p\rangle_{\Gamma_I},$$
and the standard density argument gives the desired
$$k\nabla p\cdot \mathbf e_3 = [\bv-\bw]\cdot \mathbf e_3 \in H^{-1/2}(\Gamma_I).$$
Correspondingly, we may produce $\widetilde{\bv} \in \mathbf H^1(\Omega_f)$ such that
$$\widetilde{\bv}\big|_{\partial\Omega_f} = \begin{cases} g \mathbf e_3 & \text{ on }~\Gamma_I \\ 0 & \text{ on }~\partial\Omega_f\setminus\overline{\Gamma_I},\end{cases}$$
and we  consider $g \in H_0^{1/2+\epsilon}(\Gamma_I)$. Applying this particular $\widetilde{\bv}$ to (3.71) yields
$$-\langle [\sigma_f(\bv,\pi)\mathbf e_3]\cdot \mathbf e_3, g\rangle_{\Gamma_I} = -\langle p, g\rangle_{\Gamma_I},$$
and so by density argument, 
$$[\sigma_f(\bv,\pi)\mathbf e_3]\cdot \mathbf e_3 = p \in H^{-1/2}(\Gamma_I).$$
Analogously, we obtain
$$-[\sigma_f(\bv,\pi)\mathbf e_3]\cdot\btau = \beta[\bv-\bw]\cdot \btau \in H^{-1/2}(\Gamma_I).$$
With the boundary conditions in hand, we may lastly recover $\bu$ (since we have already recovered $\bw, p, \bv$). This is done directly: reading off of \eqref{lmess} provides
\begin{equation}
\begin{cases} 
-\text{div}~\sigma^E(\bu) = -\alpha \nabla p + \rho_b\bw_2^* \\ 
\sigma^E(\bu)\mathbf e_3 = - \sigma_f(\bv,\pi)\mathbf e_3+\alpha p \mathbf e_3\\
\bu = 0 ~\text{ on }~\Gamma_b \\ 
\bu ~\text{ is  periodic on }~ \Gamma_{\text{lat}}.
\end{cases}
\end{equation}
Solving this boundary value problem, we obtain the estimate,
\begin{equation}\label{estt4} ||\bu||_{\mathbf H^1_{\#,*}(\Omega_b)}+||\sigma^E(\bu) \mathbf{n}||_{\mathbf H_{\#}^{-1/2}(\partial \Omega_b)} \le C ||\by^*||_X.\end{equation}
(in particular, the periodicity of the flux is obtained in the same way as that for the Stokes and Biot pressure traces; see \cite[Section 4.2.2]{oldpaper}).
With $\by = [\bu,\bw,p,\bv]^T \in \mathcal D(\mathcal L) = \mathcal S$ we have also obtained  harmonic $\pi \in L^2(\Omega_f)$, with boundary values from \eqref{pisbound} satisfying \eqref{pis3}. Thus, taking into account estimates \eqref{estt1}, \eqref{estt2}, \eqref{estt3}, and \eqref{estt4}, we see that the system \eqref{lmess} is well-posed---with associated stability estimate
$$||\by||_X \le C||\by^*||_X.$$
We are then justified in writing
$$\by = \mathcal L^{-1}\begin{bmatrix} \bw_1^*\\ \bw_2^*\\ q^*\\ \mathbf f^*\end{bmatrix} \in \mathcal D(\mathcal L),$$
which completes the proof of Lemma \ref{resolvent}. 
\end{proof}

Now, we complete the proof of Proposition \ref{thistheadjoint} by showing that $\mathcal D(\mathcal A^*) \subseteq \mathcal S$. 
\begin{proof}
Let $\by = [\bu,\bw,p,\bv]^T \in \mathcal D(\mathcal A)$ be given. Also let $\mathbf x \in \mathcal D(\mathcal A^*)$ be given, and set 
$$\by^*= \begin{bmatrix} \bw_1^* \\ \bw_2^* \\q^* \\ \mathbf f^*\end{bmatrix} \equiv \mathcal A^* \mathbf x $$ then write
$$\byt = \begin{bmatrix} \widetilde{\bu} \\ \widetilde{\bw} \\  \widetilde p \\  \widetilde{\bv} \end{bmatrix} \equiv \mathcal L^{-1}\by^*,$$
so $$\byt = \mathcal L^{-1}\mathcal A^*\mathbf x.$$

It is direct to verify that for given $\mathbf x \in \mathcal D(\mathcal A^*)$, with $\widetilde{\by} = \mathscr L^{-1}\mathcal A^*\mathbf x$,
\begin{align*}
(\mathcal A\by , \widetilde \by)_X = & ~ (\sigma^E(\bu), \mathbf D(\mathbf w_1^*))_{\Omega_b} + \rho_b(\bw, \bw^*_2)_{\Omega_b}+ c_0(p, q^*)_{\Omega_b} + \rho(\bv, \mathbf f^*)_{\Omega_f} \\
=& (\by,\mathcal A^*\mathbf x)_X,
\end{align*}
for all $\by \in \mathcal D(\mathcal A)$. We omit these calculations, as they closely mirror those in the preceding lemmata. 

In other words, 
$$(\mathcal A \by, \mathbf x-\mathcal L^{-1}\mathcal A^*\mathbf x)_X = 0,~~\forall~\by \in \mathcal D(\mathcal A).$$
The same proof as in Lemma \ref{resolvent} yields that $\mathcal A$ is boundedly invertible on $X$, which then, ranging through $X = \mathcal A\mathcal D(\mathcal A)$, yields that 
$$\mathbf x = \mathscr L^{-1}\mathcal A^* \mathbf x.$$ By construction, $\mathscr L^{-1}\mathcal A^* \mathbf x \in \mathcal D(\mathcal L) = \mathcal S$. But $\mathbf x$ was assumed to be in $\mathcal D(\mathcal A^*)$, so we conclude that $\mathcal D(\mathcal A^*) \subseteq \mathcal S$, as desired. Thus, $\mathcal L = \mathcal A^*$. 
\end{proof}
We have shown that $\mathcal L = \mathcal A^*$, and hence characterized the action and domain of $\mathcal A^*$. This completes the proof of Proposition \ref{thistheadjoint}.

\end{document}